\documentclass[12pt]{amsart}

\title{Ordered structures and large conjugacy classes}

\usepackage{amsfonts,amssymb,amsmath,amsthm}
\usepackage{url}
\usepackage{xcolor}
\usepackage{bbm}
\usepackage{anysize}

\usepackage{float}
\usepackage{graphicx,hyperref}
\usepackage{tikz,multicol,manfnt}

\marginsize{1 in}{1 in}{1 in}{1 in}

\theoremstyle{plain}
\newtheorem{theorem}{Theorem}[section]
\newtheorem{lemma}[theorem]{Lemma}
\newtheorem{corollary}[theorem]{\bf Corollary}
\newtheorem{remark}[theorem]{\bf Remark}
\newtheorem{example}[theorem]{\bf Example}
\newtheorem{proposition}[theorem]{\bf Proposition}

\newcommand \SB{\mathcal{SB}}
\newcommand{\fra}{Fra\"\i ss\'e }

\newcommand \Aut{ \mbox{Aut}}
\newcommand \Cone{ \mbox{Cone}}

\newcommand \SI{ S_\infty}
\newcommand \dom{{\rm{ dom}}}
\newcommand \rng{{\rm{ rng}}}
\newcommand \age{{\rm{ Age}}}
\newcommand \aut{{\rm{ Aut}}}

\newcommand \meet{{\rm{ meet}}}
\newcommand \htt{{\rm{ ht}}}

\def\f{\mathcal{F}}

\def\p{\mathcal{P}}

\def\rest{\restriction}

\newcommand \supp{{\rm{ supp}}}

\newcommand \zdef{{\rm{ def}}}

\newcommand \diam{{\rm{ diam}}}
\newcommand \qftp{{\rm{ qftp}}}

\def\f{\mathcal{F}}
\def\rest{\restriction}

\newcommand \NN{\mathbb{N}}

\newcommand \ZZ{\mathbb{Z}}

\newcommand \OOc{\mathcal{O}}

\newcommand \KK{\mathcal{K}}

\def\f{\mathcal{F}}

\def\rest{\restriction}

\newcommand\bP{\mathcal{P}}

\author[A. Kwiatkowska]{Aleksandra Kwiatkowska}
\address{Institut f\"{u}r Mathematische Logik und Grundlagenforschung, Universit\"{a}t  M\"{u}nster,  
Einsteinstrasse 62,
48149  M\"{u}nster,
Germany {\bf{and}} 
Instytut Matematyczny, Uniwersytet Wroc{\l}awski,  pl. Grunwaldzki 2/4, 50-384 Wroc{\l}aw, Poland}
\email{kwiatkoa@uni-muenster.de}

\author[M. Malicki]{Maciej Malicki}
\address{Department of Mathematics and Mathematical Economics, Warsaw School of Economics, al. Niepodleg{\l}o\'sci 162, 02-554, Warsaw, Poland}
\email{mamalicki@gmail.com}

\thanks{The first 
named author was supported by Narodowe Centrum Nauki grant 2016/23/D/ST1/01097.}

\begin{document}
	
\maketitle

\begin{abstract}
This article is a contribution to the following problem: does there exist a Polish non-archimedean group (equivalently: automorphism group of a \fra limit) that is extremely amenable, and has ample generics. As \fra limits whose  automorphism groups  are extremely amenable must be ordered, i.e., equipped with a linear ordering, we focus on ordered \fra limits.  We prove that automorphism groups of the universal ordered boron tree, and the universal ordered poset have a comeager conjugacy class but no comeager $2$-dimensional diagonal conjugacy class. We also provide general conditions implying that there is no comeager conjugacy class, 
comeager $2$-dimensional diagonal conjugacy class or non-meager $2$-dimensional topological similarity class in the automorphism group of an ordered \fra limit. We provide a number of applications of these results.
\end{abstract}

\section{Introduction}
This article is a contribution to the following question: does there exist a  Polish non-archimedean group, i.e., a Polish group with a neighborhood basis at the identity consisting of open subgroups, that simultaneously satisfies two frequently studied properties: it is extremely amenable, and it has ample generics.

A Polish group $G$ is {\em extremely amenable} if every continuous action of $G$ on a compact space has a fixed point. The group $G$ has {\em ample generics} if, for every $n \geq 1$, there exists an \emph{$n$-dimensional diagonal conjugacy class} in $G$, i.e., a set of the form
\[ \{ (gg_1g^{-1}, \ldots,gg_ng^{-1}) \in G^n\colon g \in G \}, \]
for some $g_1, \ldots, g_n \in G$, which is comeager in $G^n$. Such a group admits only one Polish group topology, and all of its (abstract) homomorphisms into separable groups are continuous (Kechris-Rosendal \cite{KR}.) In particular, \emph{every} action by homeomorphisms of an extremely amenable group with ample generics on a compact separable space has a fixed point.

It is known that there exist Polish groups sharing both of these features. Pestov-Schneider \cite{PS} proved that, for any Polish group $G$, the group $L_0(G)$, i.e., the group of measurable functions with values in $G$, is extremely amenable, provided that $G$ is amenable, and Ka\"{i}chouh-Le Ma\^{i}tre \cite{KLM} proved that $L_0(G)$ has ample generics whenever $G$ has. As $\SI$, i.e., the group of all permutations of natural numbers, is amenable, and has ample generics, $L_0(\SI)$ is extremely amenable and it has ample generics. However, it is still an open problem whether there are such groups in the non-archimedean realm.

Let $M$ be a first-order countable structure.
It is well known that every Polish non-archimedean group is isomorphic to the automorphism group $\Aut(M)$ of a structure $M$ (i.e., a set equipped with relations and functions) equal to the \fra limit of a \fra class $\f$ of finite structures  (see the next section for precise definitions of notions used in the introduction.) 
The group $\Aut(M)$ naturally acts 
on the compact space of linear orderings of  $M$, viewed as a subspace of   $\{0,1\}^{M\times M}$. 
This implies that when $\aut(M)$ is
extremely amenable, then there is a linear ordering of $M$ preserved by $\Aut(M)$, see also \cite{KPT}.
Therefore if $\Aut(M)$ is extremely amenable, we can actually  assume that $\f$ is an  order class, i.e., that each structure in $\f$ is  equipped with a linear ordering $<$  of its elements. Thus, we can pose a more general question: does there exist a \fra limit $M$ of an order class $\f$ such that the automorphism group $\Aut(M)$ has ample generics. This article gives some partial answers as to when such a situation cannot happen.

Curiously enough, there are no known examples of Polish groups that do not have ample generics but they have a comeager diagonal conjugacy class for some $n \geq 2$. Thus, our article can be also viewed as a study of the question whether comeager diagonal conjugacy classes resemble weak mixing in topological dynamics, which implies weak mixing of all orders (see \cite{F}.)

One of our main tools is a theorem of Kechris-Rosendal, connecting the structure of diagonal conjugacy classes in the automorphism group of the \fra limit $M$ of a \fra class $\f$ with the joint embedding property (JEP), and the weak amalgamation property (WAP) in classes $\f_n$ of $n$-tuples of partial automorphisms of elements of $\f$.
 They prove (see also \cite{I}) that $\Aut(M)$ has a comeager $n$-dimensional diagonal conjugacy class if and only if $\f_n$ has JEP and WAP. Thus, showing that $\Aut(M)$ does not have a comeager $n$-dimensional diagonal conjugacy class reduces to verifying that $\f_n$ has no JEP or WAP.

First, we study the one-dimensional case.  A class of structures $\f$ has the {\em $1$-Hrushovski property} if every partial automorphism of an $A\in\f$ can be extended to an automorphism of some $B\in\f$. Clearly, if $\f$ is an order class of finite structures, then $\f$ does not have the $1$-Hrushovski property because in this case non-trivial orbits are necessarily infinite. We introduce the notions of strong splitting and always strong splitting in a \fra class, which capture the idea of `flexible' amalgamation. Then we prove (Theorem \ref{th:NoComeager1}) that $\f_1$ has no WAP, provided that one of following holds: $\f$ is a \fra class that does not have the $1$-Hrushovski property, and it always strongly splits, or $\f$ is a full order expansion of $\mathcal{K}$ (i.e., $\f$ is the class of \emph{all} linear orderings on elements of $\KK$), where $\mathcal{K}$ is a \fra class that strongly splits. On the other hand, we show (Theorem \ref{boron}) that the class $\SB_1$ of partial automorphisms of ordered boron trees, and (Theorem \ref{wapposet}) the class $\mathcal{P}_1$ of partial automorphisms of ordered partial orders, have CAP, so, in particular, they have WAP. It seems that these are, except for $\Aut(\mathbbm{Q})$ (see Truss  \cite{T}), the only known  order classes such that the automorphism group of the limit has a comeager conjugacy class.
We also give (Theorem \ref{th:QU}) a short and elementary proof of a result of Slutsky \cite{Sl} who showed that the class $(\mathcal{QU}_\prec)_1$ of partial automorphisms of ordered metric spaces with rational distances has no WAP.   

Next, we turn to the two-dimensional case. For a \fra class $\f$, we formulate a simple but efficient condition (Proposition \ref{nowap}) implying that $\f_2$ has no WAP, and we verify it for a number of cases such as 
precompact Ramsey expansions of ultrahomogeneous directed graphs, in particular for $\mathcal{P}_2$.   Using a similar approach, we also show (Theorem \ref{nowapbor}) that $\SB_2$ does not have WAP. Then we investigate topological similarity classes. For a Polish group $G$, $n \geq 1$, and an $n$-tuple $(f_1, \ldots, f_n)$ in $G$, the \emph{$n$-dimensional topological similarity class} of $(f_1, \ldots, f_n)$ is the family of all $n$-tuples $(g_1, \ldots, g_n)$ in $G$ such that the mapping $f_i \mapsto g_i$ (uniquely) extends to a topological group isomorphism. Clearly, this is a generalization of the notion of the diagonal conjugacy class, and it is still not known whether there exists a Polish group $G$ such that for some $n \geq 2$ there is  a non-meager $n$-dimensional topological similarity class, but all $n$-dimensional diagonal conjugacy classes are meager. Generalizing methods and results of Slutsky \cite{Sl}, we show (Theorem \ref{th:MeagerSim}) that if $M$ is the \fra limit of a \fra class $\f$ that is a full order expansion and that satisfies certain additional conditions, then all $2$-dimensional topological similarity classes in $\Aut(M)$ are meager. In particular, this is true if $\mathcal{K}$ is a class with free amalgamation, or the class of ordered tournaments (Theorem \ref{th:MeagerSimFree}.)   
 
\section{Definitions}

A topological group is {\em Polish} if its topology is separable, and completely metrizable. A Polish group is {\em non-archimedean} if it has a neighborhood basis at the identity consisting of open subgroups, or, equivalently, it is topologically isomorphic to the automorphism group $\Aut(M)$ of a countable structure, equipped with the product topology (i.e., $\Aut(M) \subseteq M^M$, where $M$ is regarded as a discrete space.)
 
By a structure we always mean a relational structure (i.e., a set equipped with relations), and we consider only classes of finite structures.  Let $A$ be a structure, and let $B,C \subseteq A$. By $\qftp_A(B/C)$, we denote the quantifier-free type of $B$ over $C$ in $A$. Let $p$ be a partial automorphism of~$A$. We write $\zdef(p)=\dom(p) \cup \rng(p)$, and $\supp(p)=\{x \in \zdef(p)\colon p(x) \neq x \}$. An \emph{orbit} of $p$ is a maximal subset $\OOc \subseteq A$ that can be enumerated into $\{a_0, \ldots, a_m\}$ so that $p(a_i)=a_{i+1}$ for $i<m$. If $p(a_m)=a_0$, we say that $\OOc$ is a \emph{cyclic orbit}. 
An orbit is {\em trivial} if it consists of a single element.

Let $\f$ be a class of structures in a given signature. We say that $\f$ has \emph{JEP} (the joint embedding property) if any two $A,B \in \f$ can be embedded in a single $C \in \f$. We say that $\f$ has {\em AP}  (the amalgamation property) if for every $A, B,C \in\f$ and embeddings $\alpha\colon A\to B$ and $\beta\colon A\to C$ there is $D\in\f$ and embeddings $\gamma\colon B\to D$, $\delta\colon C\to D$ such that $\gamma\circ\alpha=\delta\circ\beta$. In that case, we say that $B$ and $C$ {\em amalgamate } over $A$. We say that $\f$ has {\em SAP}  (the strong amalgamation property) if, additionally, $\gamma[B] \cap \delta[C]=\gamma\circ\alpha[A]$. We say that $\f$ has CAP (the cofinal amalgamation property) if there is a cofinal (with respect to inclusion) subclass of $\f$ with AP. 
We say that $\f$ has {\em WAP} (the weak amalgamation property) if for every $A\in\f$ there is $A'\in\f$ and an embedding $\phi\colon A\to A'$ such that for every $B,C\in\f$ and embeddings $\alpha\colon A'\to B$, $\beta\colon A'\to C$ there is $D\in\f$ and embeddings $\gamma\colon B\to D$, $\delta\colon C\to D$ such that $\gamma\circ\alpha\circ\phi=\delta\circ\beta\circ\phi$. Clearly, if $\f$ has CAP, then it has 
WAP. Actually, in the definition of AP (CAP and WAP), it suffices to consider $B,C$ such that $B \cap C=A$ ($B \cap C=A'$), and only trivial embeddings, i.e., inclusions.  
 
A class of finite structures $\f$ is  a \emph{\fra class}, if it is countable (up to isomorphism), closed under isomorphism, closed under taking substructures, and has JEP and AP.  
A countable  first-order structure $M$ is {\em ultrahomogeneous} if every isomorphism between finite substructures of $M$ can be extended to an automorphism of the whole $M$. Then  $\rm{Age}(M)$ -- the class  of all finite substructures embeddable in $M$ -- is a {\em \fra\ class}. Conversly, by the classical theorem due to Fra\"\i ss\'e,  for every \fra\ class  $\mathcal{F}$ of finite structures, there is a unique up to isomorphism countable ultrahomogeneous structure $M$ such that  $\mathcal{F}=\age(M)$. 
We call this  $M$ the  {\em Fra\"{i}ss\'{e} limit} of $\mathcal{F}$.

A \fra class $\f$ is called an \emph{order class} if its signature includes a binary relation defining a linear ordering on each element of $\f$. If $\f$ is an order class, $\f^-$ denotes the reduct of  $\f$ obtained by removing the order relation $<$ from the signature of $\f$. We call $\f$ a \emph{full order expansion} if $\f=\f^-* \mathcal{LO}$, i.e., it is a class of elements of the form $(A,<)$, where $A \in \f^-$, and $<$ is \emph{any} linear ordering of $A$. We will frequently use the observation that if $M$ is the \fra limit of a full order expansion with SAP, then for any finite $A \subseteq B \subset M$, and any $<$-interval $I$ in $M$, there exists $C \subseteq I$ that is isomorphic to $B$ via an isomorphism that pointwise fixes $A$. 


We will be mostly interested in classes of tuples of partial automorphisms of structures coming from a given class $\f$. Formally, for $n \geq 1$, denote
\[ \f_n=\{(A,p_1,\ldots,p_n)\colon A\in\f, \ p_i \text{ is a partial automorphism of } A, \ i \leq n \}.\]
Often, we will think of elements of $\f_n$ simply as tuples of partial automorphisms. Then $(p_1,\ldots,p_n)$ is identified with $(\bigcup_i \zdef(p_i),p_1, \ldots, p_n)$. A map  $\phi\colon (A,p_1,\ldots,p_n)\to (B,q_1,\ldots,q_n)$ will be called an {\em embedding} if 
it is an embedding of $A$ into $B$, and  $\phi\circ p_i =   q_i \circ\phi$ for $i \leq n$. Using this notion of embedding,  we can also define properties JEP, AP, CAP, and WAP for classes $\f_n$. Then we have:

\begin{theorem}[Kechris-Rosendal \cite{KR}]\label{kechros}
Let $\f$ be a \fra class, and let $M$ be the \fra limit of $\f$.
\begin{enumerate}
\item There exists a dense diagonal $n$-conjugacy class in $\aut(M)$ iff $\f_n$ has JEP, 
\item there exists a comeager diagonal $n$-conjugacy class in $\aut(M)$ iff $\f_n$ has JEP and WAP.
\end{enumerate}
\end{theorem}

In particular, it follows that if $\f_n$ has JEP but no WAP, then $\Aut(M)$ has meager $n$-dimensional diagonal conjugacy classes.

Let $F_n= F_n(s_1,\ldots,s_n)$ denote the free group on $n$ generators $s_1,\ldots,s_n$. For a word
$w\in  F_n$, and an $n$-tuple $\bar{f}=(f_1,\ldots, f_n)$ in $G$, the evaluation $w(\bar{f})$ denotes the element of $G$ obtained from $w$
by substituting $f_i$ for $s_i$, and performing the group operations on the resulting sequence.
By a word, we will always mean a reduced word.

\section{The one-dimensional case. Conjugacy classes}

\subsection{A condition that implies the failure of WAP}\label{dim1nowap}
Recall that a family $\f$  of finite structures in a given signature has the {\em  Hrushovski property} if for every 
$n\in\mathbb{N}$, $A\in\f$ and a tuple $(f_1,\ldots,f_n)$ of partial automorphisms of~$A$, there exists $B\in \f$ such that $A\subseteq B$, and every $f_i$ can be extended to an automorphism of $B$. 
We say that $\f$ has the {\em $n$-Hrushovski property} if the above holds for a given $n$.

In \cite[Theorem 4.7]{KwMa}, we proved the following trichotomy.

\begin{theorem}\label{47}
	\label{th:Hrushovski}
	Let $M$ be a \fra limit of a \fra family $\f$ such that algebraic closures of finite subsets of $M$ are finite. Then one of the following holds:
	\begin{enumerate}
		\item $\f$ has the Hrushovski property,
		\item $\f$ does not have the $1$-Hrushovski property,
		\item there exists $n$ such that none of $n$-dimensional topological similarity classes in $\Aut(M)$ is comeager. In particular, $\Aut(M)$ does not have ample generics.
	\end{enumerate}
\end{theorem}


It is well known that if a \fra class $\f$ has the Hrushovski property, and sufficiently free amalgamation, then the automorphism group $\Aut(M)$ of its limit $M$ has ample generics. 
By the above trichotomy, if $\f$ does not have the Hrushovski property, ample generics may be present in $\Aut(M)$ only if $\f$ does not even have the $1$-Hrushovski property -- which is true, in particular, for order classes. In this section, we prove (in Theorem \ref{th:NoComeager1}) that such situations always presuppose a very rigid form of amalgamation in $\f$. In order to specify what `rigid' is about in this context, let us introduce two definitions that capture what `flexible' amalgamation means for us.

In \cite[Definition 2.4]{Pa}, Panagiotopolus studies extensions of automorphisms of generic substructures of a given structure. He introduces the notion of splitting in a \fra class $\f$. An element $C \in \f$ \emph{splits} $\f$ if for every $D \in \f$ with $C \subseteq D$ there exist $D_1,D_2 \in \f$ with $D \subsetneq D_1,D_2$, and a bijection $f\colon D_1 \rightarrow D_2$ such that 

\begin{enumerate}
\item $f$ pointwise fixes $D$,
\item $f \rest (D_1 \setminus C)$ is an isomorphism between $D_1 \setminus C$ and $D_2 \setminus C$,
\item $f$ is not an isomorphism between $D_1$ and $D_2$.
\end{enumerate}

Analogously, we will say that $C \in \f$  \emph{strongly splits} $\f$ if for all $D,D_1 \in \f$ with $C \subseteq D \subsetneq D_1$ there exists $D_2 \in \f$ with $D \subsetneq D_2$, and a bijection $f\colon D_1 \rightarrow D_2$ such that Conditions (1)-(3) above hold. We will say that $\f$ \emph{strongly splits} if there exists $C \in \f$ that strongly splits $\f$, and that $\f$ \emph{always strongly splits} if every $C \in \f$ strongly splits $\f$.

We can think of $C$ in the above definitions as one `ear' of an amalgamation diagram $U \subsetneq V,W$, i.e., $C=V \setminus U$ and $D=V$. Then $C$ strongly splits if for any other `ear' $W \setminus U$ (i.e., $D_1 \setminus D$), there are at least two non-equivalent ways in which we can define relations involving elements from the `ears' $V \setminus U$ and $W \setminus U$ to form an amalgam of $V$ and $W$ over $U$: one represented by $D_1$ (where $W=D_1 \setminus C$), the other one by $D_2$ (where $D_2 \setminus C$ is an isomorphic copy of $W$.)

In particular, the property of always strong splitting can be also expressed as a variant of the amalgamation property: $\f$ always strongly splits if amalgamation in $\f$ is not too rigid, that is, if there is always more than one way of amalgamating structures. To be more precise, fix $C$, $D$ and $D_1$ as above, and think of $D_1$ as an amalgam of $D$ and $D_1 \setminus C$ over $D\setminus C$. Then any other non-isomorphic amalgam with the same underlying sets gives a required $D_2$. In other words, a class $\f$ always strongly splits if for any $A,U,V \in \f$ with $A< U,V$ there exist two non-isomorphic amalgams $W_1, W_2$ of $U$ and $V$ over $A$, with $U \cup V$ as the underlying set, and such that $U$ and $V$ embed in both $W_1$ and $W_2$ by the identity mapping.  

Proving the next two lemmas is straightforward, and left to the reader.

\begin{lemma}
\label{le:ReasonExp}
If $\f$ is a full order expansion of a class that (always) strongly splits, then $\f$ also (always) strongly splits.
\end{lemma}

\begin{lemma}
\label{le:FlexExt}
Let $p$ be a partial automorphism of a structure $A$, and let $x \in \rng(p)\setminus \dom(p)$. Suppose that $y,y' \in A$ are such that $p \cup (x,y)$ is a partial automorphism, and $\qftp_A(y/\rng(p))=\qftp_A(y'/\rng(p))$. Then $p \cup (x,y')$ is also a partial automorphism of~$A$.
\end{lemma}

\begin{lemma}
\label{le:SepOrb}
Let $\f$ be a full order expansion with SAP. Then for every $p \in \f_1$, there exists $q \in \f_1$ extending $p$ such that for every $r \in \f_1$ extending $q$, distinct orbits of $p$ are contained in distinct orbits of $r$.
\end{lemma}

\begin{proof}
Fix $p \in \f_1$, and let $\OOc_0, \ldots, \OOc_n$ be orbits of $p$. 
Fix $i<j \leq n$, and suppose that there exists an extension $q'$ of $p$ such that $\OOc_i$ and $\OOc_j$ are in the same orbit of $q'$. Then there must exist a partial automorphism $q''$ extending $p$, and $x, y \in \rng(q'') \triangle \dom(q'')$, $x$ in the orbit of $q''$ determined by $\OOc_i$, $y$ in the orbit of $q''$ determined by $\OOc_j$, and we can extend $q''$ by putting $q''(x)=y$ or $(q'')^{-1}(x)=y$.
Without loss of generality, we can assume that $x<y$, and $q'' \cup (x,y)$ extends $q''$. Since $\f$ is a full order expansion with SAP, there exists $C \in \f$ with $\zdef(q'') \subseteq C$, and $y' \in C \setminus \zdef(q'')$ with $y'>y$, and such that $\qftp_C(y/\rng(q''))=\qftp_C(y'/\rng(q''))$. But then, by Lemma \ref{le:FlexExt}, $q=q'' \cup (x,y')$ also extends $q''$, and $x<y<q(x)$. Clearly, $\OOc_i$ and $\OOc_j$ stay distinct in any extension of $q$. By iterating this construction, we can find $q$ that works for all $i<j \leq n$.
\end{proof}

\begin{theorem}
\label{th:NoComeager1}
Let $\f$ be a \fra class. Suppose that

\begin{enumerate}
\item $\f$ does not have the $1$-Hrushovski property, and it always strongly splits, or 
\item $\f$ is a full order expansion with SAP, and it strongly splits.
\end{enumerate}
Then $\f_1$ has no WAP.
\end{theorem}

\begin{proof}

Assume that Condition (1) holds. Fix $p \in \f_1$ witnessing that $\f$ does not have the $1$-Hrushovski property. We show that $p$ also witnesses that $\f_1$ does not have WAP. 

Fix $q \in \f_1$ that extends $p$. Clearly, there must exist an orbit $\mathcal{O}$ of $q$ intersecting $\dom(p)$ that is non-cyclic -- otherwise the union of such orbits would be a structure in $\f$ invariant under $q$. Let $\OOc=\{o_0, \ldots, o_n\}$ with $q(o_i)=o_{i+1}$ for $i<n$. As $\OOc$ is not cyclic, $q$ is not defined on $o_n$. Fix $y_0 \not \in \zdef(q)$ such that $q_0=q \cup (o_n,y_0)$ is a partial automorphism of $D_1=\zdef(q) \cup \{y_0\}$. 
Since $\f$ always strongly splits, by putting $C=\dom(q) \setminus \rng(q)$ (note that $o_0 \in C$, so $C \neq\emptyset$), $D=\zdef(q)$ (and $D_1=\zdef(q) \cup \{y_0\}$), we can find $y_1$ such that $D_2= \zdef(q) \cup \{y_1\}$ witnesses that $C$ strongly splits. However, this means that 
\begin{equation}
\label{eq:00}
\qftp_C(y_0/ \rng(q))=\qftp_C(y_1/\rng(q))
\end{equation}
but
\begin{equation}
\label{eq:0}
\qftp_C(y_0/(\zdef(q))) \neq \qftp_C(y_1/ (\zdef(q))).
\end{equation}
Then (\ref{eq:00}) together with Lemma \ref{le:FlexExt} implies that $q_1=q \cup (o_n,y_1)$ is also a partial automorphism. On the other hand, for every $r \in \f_1$ such that $q_0$ and $q_1$ can be embedded into $r$ by embeddings $e_0$ and $e_1$, respectively, that agree on $\zdef(q)$, we must have $e_0(y_0)=e_1(y_1)$, and this is impossible because of (\ref{eq:0}). Moreover, the same argument can be applied to any extension of $p$, so, in fact, if $e_0$, $e_1$ agreed on $\zdef(p)$, they would agree on $\zdef(q)$ as well. Thus, $q_0$, $q_1$ cannot be amalgamated over $p$. As $q$ was  arbitrary, $p$ witnesses that $\f_1$ does not have WAP. 

Assume now that Condition (2) holds. Fix $A \in \f$ witnessing that $\f$ strongly splits. Fix $p \in \f_1$ such that $A=\dom(p) \setminus \rng(p)$ (this can be easily done using the assumption that $\f$ is is a full order expansion with SAP), and extend $p$ to a partial automorphism $q$ as in Lemma \ref{le:SepOrb}. Then for every extension $r \in \f_1$ of $q$, there exists $A' \subseteq \dom(r) \setminus \rng(r)$ that is isomorphic with $A$, and so $A'$ also witnesses that $\f$ strongly splits. Now we proceed as in the proof of  (1). Fix a (non-cyclic) orbit $\OOc=\{o_0, \ldots, o_n \}$ of $r$ intersecting $\dom(p)$, find $y_0$ such that $r_0=r \cup \{(o_n,y_0)\}$ is a partial automorphism, and put $C=A'$, $D=\zdef(r)$, $D_1=\zdef(r) \cup \{y_0\}$ to obtain $y_1$ such that $r_0$ and $r_1=r \cup (o_n,y_1)$ cannot be amalgamated over $p$. 
\end{proof}


\begin{proposition}
\label{le:TourGraphSplit}
The classes of ordered graphs, and ordered tournaments always strongly split. 
\end{proposition}

\begin{proof}
Let $\f$ be either the class of ordered graphs or ordered tournaments. Fix $C, D, D_1 \in \f$ with $C \subseteq D \subsetneq D_1$. Fix $c \in C$ and $d_1 \in D_1 \setminus D$. For graphs, define $D_2$ to be the graph that differs from $D_1$ only in that $\{ c,d_1\}$ is an edge in $D_2$ if and only if $\{ c,d_1\}$ is not an edge in $D_1$. Similarly, for ordered tournaments, define $D_2$ to be the ordered tournament that differs from $D_1$ only in that $(d_1,c)$ is an arrow in $D_2$ if and only if $(c,d_1)$ is an arrow in $D_1$.

\end{proof}

\begin{corollary}
\begin{enumerate}
\item The class of partial automorphisms of finite ordered tournaments has no WAP.
\item (Slutsky) The class of partial automorphisms of finite ordered graphs has no WAP.
\end{enumerate}
\end{corollary}

\begin{proof}
Finite ordered graphs and finite ordered tournaments are full order expansions with SAP. By Proposition \ref{le:TourGraphSplit}, both of them always strongly split, so, by Theorem \ref{th:NoComeager1}, they have no WAP.
\end{proof}

On the other hand, it is easy to see that both of these classes have JEP, which means (by Theorem \ref{kechros}), that automorphism groups of the universal ordered tournament and the universal ordered graph (i.e., \fra limits of the above classes) have meager conjugacy classes.

Slutsky \cite{Sl} also proved that the automorphism group $\Aut(\mathbbm{QU}_{\prec})$ of the ordered rational Urysohn space $\mathbbm{QU}_{\prec}$, i.e., the full order expansion of finite ordered metric spaces with rational distances, has meager conjugacy classes. One of the ingredients of his proof is a deep theorem of Solecki saying that the class of finite metric spaces has the Hrushovski property. Theorem \ref{th:NoComeager1} cannot be used to recover Slutsky's result because the class of ordered finite metric spaces with rational distances does not strongly split. However, a similar approach gives rise to a more elementary argument. We sketch it below. Note that $(\mathcal{QU}_{\prec})_1$ has JEP,
therefore it will suffice to prove that it has no WAP.

\begin{theorem}[Slutsky]
\label{th:QU}
The class $(\mathcal{QU}_{\prec})_1$ has no WAP.
\end{theorem}

\begin{proof}
First fix $A \in \mathcal{QU}$, and $x \in A$. Without loss of generality, we can assume that $A \subseteq \mathbbm{QU}$. Suppose that $y \in \mathbbm{QU}$ is such that the type $\qftp_A(y/(A\setminus \{x\}))$ determines $d(x,y)$. By the triangle inequality, this is possible only when there are $a,a' \in A$ such that
\[ d(y,x)=d(y,a)+d(a,x), \]
\[ d(a',x)=d(a',y)+d(y,x). \]
But then, in particular, $d(x,y) \leq \diam(A)$. Thus, if for some partial automorphism $p$ there was an automorphism $q$ extending $p$ such that any two extensions $r_0$, $r_1$ of $q$ could be amalgamated over $p$, then for every automorphism $\phi$ of $\mathbbm{QU}_{\prec}$ extending $q$, orbits of $\phi$ determined by $p$ would be bounded by $\diam(\zdef(q))$. But it is well known (see, e.g., Section 3.1 in \cite{DoMa}) that every partial isometry $q$ of $\mathbbm{QU}$ with no cyclic orbits can be extended to an isometry of $\mathbbm{QU}$ with unbounded orbits. And since $\mathcal{QU}_{\prec}$ is a full order expansion of $\mathcal{QU}$, every partial automorphism of a finite subspace of $\mathbbm{UQ}_{\prec}$ also can be extended to an automorphism of $\mathbbm{QU}_{\prec}$ with unbounded orbits.
\end{proof}

\begin{remark}
Using a similar approach, and a construction as in the proof of Lemma \ref{le:FreeAm}, one can also prove that the class of partial automorphisms of ordered $K_n$-free graphs does not have WAP, for every $n \geq 3$.
It is not hard to see that this class does not strongly split.
\end{remark}

It is known that the class of partial automorphisms of finite linear orderings has JEP and WAP. In the next two sections, we present two others such order classes: ordered boron trees, and ordered posets.

\subsection{Ordered boron trees - a comeager conjugacy class}
In this section, we  prove that the automorphism group of the universal ordered boron tree 
has a comeager conjugacy class. The class of boron trees was introduced by Cameron \cite{Ca}.

Let $\mathcal{T}$ denote the class of  all graph-theoretic trees such that the valency of each vertex is equal to 1 or 3.
If $T\in\mathcal{T}$  and $a,b,c,d\in T$, we let $ab|cd$ iff arcs $ab$ and $cd$ do not intersect.
To each $T\in\mathcal{T}$ we assign a structure $(B(T), R^{B(T)})$ such that $B(T)$ is the set of endpoints of $T$, and
$R^{B(T)}(a,b,c,d)$ iff $a,b,c,d$ are pairwise different and $ab|cd$. Structures $(B(T), R^{B(T)})$,
together with the one point structure, we call {\em boron trees}, and we denote the class of all these structures by $\mathcal{B}$. The {\em universal boron tree} is the \fra limit of $\mathcal{B}$.

Let $2^{<n}$ denote the set of binary sequences of the length $<n$, including the empty sequence.
Let $T'_n$ denote the binary tree, that is, a graph with the set of vertices equal to $2^{<n}$ and edges
exactly between vertices $s$ and $si$, $i=0,1$, $s\in  2^{<n}$. Let $T_n$ be the  graph obtained by removing 
the vertex $\emptyset $ from $T'_n$ and replacing edges $[\emptyset,0]$ and $[\emptyset,1]$ by 
the edge $[0,1]$, and denote $B_n=B(T_n)$.
Let $\leq^n$ be the lexicographical order on $B_n$, i.e. we let $s\leq^n t$ iff $s=t$ or $s(i)<t(i)$, where $i$ is the least
such that $s(i)\neq t(i)$.
For $s\in 2^{<n}$ we define the {\em height } as the length of $s$, i.e. $\htt(s)=|s|$. 
For a fixed $n$ and $s,t\in T'_n$, we let $s<t$ if $s$ is an initial segment of $t$, and let for $s,t\in T'_n$ the 
{\em meet} of $s$ and $t$, denoted $\meet(s,t)$, be the least upper bound of $s$ and $t$ with respect to 
the partial order $<$. 
Let $(A, R^A)\in\mathcal{B}$ and let $\phi\colon (A, R^A)\to (B_n, R^{B_n})$ be an embedding
(with respect to $R$, this is not necessarily a graph embedding).
We let $\leq_{lex}^A$ to be the order inherited from $\leq^n$ and define a ternary relation $S^A$ on $A$ as follows:  
\[S^A(a,b,c)  \iff  \phi(a), \phi(b)<_{lex}^{n} \phi(c) \text{ and } \htt(\meet(\phi(a),\phi(b)))>\htt(\meet(\phi(b),\phi(c))).\]
There are multiple ways to expand an $(A, R^A)\in\mathcal{B}$ by adding relations $S^A$ and $<^A_{lex}$.
Intuitively, the relation $S^A$ adds a root at an edge of the tree $T$ such that $B(T)=A$, viewed as a graph.
We illustrate the construction of an expansion in Figure \ref{fig:Sandlex}.
Structures $(A,R^A,S^A,\leq_{lex}^A)$, we call {\em ordered boron trees}, and we denote the class of all these structures by $\mathcal{SB}$. The {\em universal ordered boron tree} is the \fra limit of $\mathcal{SB}$. 
 From the work of Jasi\'nski \cite{J}, it follows that the automorphism group of the 
 universal ordered boron tree is extremely amenable.

Note that for $a<_{lex} b<_{lex} c<_{lex} d$ we have $R(a,b,c,d)$ iff 
$S(a,b,c)$ or ($\neg S(a,b,c)$ and $\neg S(b,c,d)$). Therefore, we can recover $R$ from $S$ and $<_{lex}$.

\begin{figure}[H] 
		\begin{tikzpicture}[scale=0.9] 
		\draw [thick] (0,0) -- (0,1);
		\draw [thick] (0,0) -- (-.7,-.9);
			\node[left] at (-.7,-.9) {$c$};
				\draw [thick] (0,0) -- (.7,-.9);
					\node[right] at (.7,-.9) {$b$};
				\draw [thick] (0,1) -- (-.7,1.9);
					\node[above] at (-.7,1.9) {$a$};
				\draw [thick] (0,1) -- (.7,1.9);
					\draw [thick] (.7,1.9) -- (1.5,2.7);
					\node[right] at (1.5,2.7) {$e$};
					\draw [thick] (.7,1.9) -- (1.4,1);
					\node[right] at (1.4,1)  {$d$};
					\draw [fill] (.35,1.45) circle [radius=0.1];
	
	\begin{scope}[xshift=10cm]
	\draw [fill] (0,0) circle [radius=0.1];
		\draw [thick] (0,0) -- (-1,.9);
				\draw [thick] (0,0) -- (1,.9);
					\draw [thick] (1,.9) -- (.3,2.7);
					\node[above] at (.3,2.7)  {$d$};
					\draw [thick] (1,.9) -- (1.7,2.7);
					\node[above] at (1.7,2.7)  {$e$};
							\draw [thick] (-1,.9) -- (-.3,2.7);
								\node[above] at (-.3,2.7) {$c$}; 
							\draw [thick] (-1,.9) -- (-1.7,2.7);
							\node[above] at (-1.7,2.7) {$a$};
							\draw [thick] (-.65,1.8) -- (-1,2.7);
							\node[above] at (-1,2.7) {$b$};
	\end{scope}				
		\end{tikzpicture}	
		\caption{Expanding $(A=\{a,b,c,d,e\}, R^A)$ by  $S^A$ and $<^A_{lex}$} \label{fig:Sandlex}
		\end{figure}
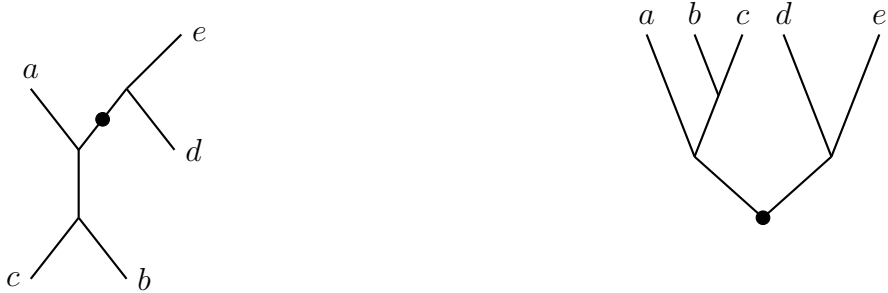

\begin{theorem}\label{boron}
	The family $\SB_1$ has CAP. 
\end{theorem}

For  $A\in\SB$, denote by $T_A$ the binary tree 
such that $A=B(T_A)$. For every $A$ there exists  unique such a $T_A$.
The {\em root} of $T_A$, denoted by $\rho_A$, is the $<$-least element of $T_A$.
By $<_A$ we denote the usual tree partial ordering of being an initial segment on elements of 
$T_A$. In the sequel, the structure $A$ in symbols $\leq_{lex}^A$, $<^A$, $R^A$, and $S^A$ will be always clear from the context, so, in order to simplify notation, we will simply write $\leq_{lex}$, $<$, $R$, and $S$,
respectively.

Let $(A,p)\in\SB_1$. 
We say that a non-trivial orbit $O=\{a_0,\ldots, a_n\}$ of $p$ is {\em increasing} if
$a_0<_{lex}  \ldots <_{lex}  a_n$; analogously, we define a {\em decreasing} orbit.
Clearly, every orbit is either increasing, decreasing, or trivial.
Note that, setting $t_i=\meet(a_i, a_{i+1})$,  we either have $t_0< t_1<t_2<\ldots< t_{n-1}$ or
$t_{n-1}< \ldots< t_2< t_1$. In the first case, we  say that $O$ is {\em meet-increasing}, and in the second
that it is {\em meet-decreasing}.
If $(B,p)\in\SB_1$ extends $(A,p)$, then we will denote by $O_B$ the extension 
of $O$ in $B$.

Let $A=(A,p)\in\SB_1$. We will call two orbits $O=\{a_0,\ldots,a_m\}$ and $P=\{b_0,\ldots, b_n\}$ of $p$ 
{\em intertwining} if the $\leq_{lex}$-intervals $(\min\{a_0,a_m\}, \max\{a_0,a_m\})_{lex}$ and 
 $(\min\{b_0,b_n\}, \max\{b_0,b_n\})_{lex}$ intersect. Say that $O$ is {\em $\leq_{lex}$-contained} in $P$ if $(\min\{a_0,a_m\}, \max\{a_0,a_m\})_{lex}$
is contained in $(\min\{b_0,b_n\}, \max\{b_0,b_n\})_{lex}$. 
A point $x\in A$ is {\em meet-locked} by $O$ if
for every extension $(B,q)$ of $(A,p)$ such that $a_{-1}=q^{-1}(a_0)$ and $a_{m+1}=q(a_m)$ are defined,
denoting $t_i=\meet(a_i, a_{i+1})$, we have the following.
(1) If $O$ is increasing and meet-increasing, then $a_{m+1}<_{lex} x$ and $t_{-1}< \meet(x, a_{m+1})<t_m$, 
(2) if $O$ is decreasing and meet-increasing, then $x <_{lex} a_{m+1}$ and $t_{-1}< \meet(x, a_{m+1})<t_m$, 
(3) if $O$ is increasing and meet-decreasing, then $x<_{lex} a_{-1}$ and $t_{-1}<\meet(x, a_{-1})<t_m$, 
(4) if $O$ is decreasing and meet-decreasing, then $a_{-1} <_{lex} x$ and $t_{-1}<\meet(x, a_{-1})<t_m$.
Two orbits $O$ and $P$ are {\em meet-intertwining}
if there is $x\in O$ meet-locked by $P$ or there is $x\in P$ meet-locked by $O$.
Note that if $O$ and $P$ are meet-intertwining then one of them is increasing and the other one is decreasing.
Moreover, either both are meet-increasing or both are meet-decreasing.

We call a {\em cone} any set  ${\rm Cone}_t=\{s\in A\colon t\leq s\}$, for some $t\in T_A$.
The  {\em root} of the orbit $O$ is the meet $t_O\in T_A$ of all points
$a_0,\ldots, a_n$ (which, in fact is the meet of two first elements in the orbit, if the orbit is meet-increasing, or the last two, if it is meet-decreasing), and the cone ${\rm Cone}_O$ of $O$ is defined as ${\rm Cone}_{t_O}$.
Note that any two cones  are either disjoint or one is contained in the other.
Denote by ${\rm Cone}(p)$ the collection of all cones of orbits of $p$. 


For $A\in\SB$, by a {\em segment} we mean an ordered pair $(x,y)$ with $x,y\in T_A$ such that $x< y$ and there is no 
$z\in T_A$ satisfying $x< z< y$. 
For  $A, E\in\SB$ and a segment $(x,y)$ in $A$, let $K=A(x,y,E,\emptyset)\in \SB$  
be the result of attaching $E$ to $A$ on $(x,y)$ on the left.
Specifically, think that elements of each of $T_A$ and $T_E$ are binary sequences, in particular, if $x=s$ and
$y=t$, then $t=s0$ or $t=s1$.
We let $T_K$ to consist of the following binary sequences.
If $r\in T_A$ does not extend properly $x=s$, we let $r\in T_K$. 
We let $t\in T_K$.
If  $tr\in T_A$ for some $r$, we let $t1r\in T_K$ and if  $r\in T_E$, let  $t0r\in T_K$. This defines $K\in\SB$.
Analogously, define $K=A(x,y,\emptyset, F)\in \SB$ as the result of attaching $E$ to $A$ on $(x,y)$ on the right.
More generally, for  $A \in \mathcal{SB}, \mathcal{E}=(E_1,\ldots, E_m)$ and $\mathcal{F}=(F_1,\ldots,F_n)$ 
and a segment $(x,y)$ in $A$, we define 
$K=A(x,y,\mathcal{E}, \mathcal{F})\in \SB$ as the result of attaching  $E_1, \ldots, E_m$
to the segment $(x,y)$ on the left in a way that $E_1<_{lex} \ldots <_{lex} E_m$ and
attaching $F_1, \ldots, F_n$
to the segment $(x,y)$ on the right in a way that $F_1<_{lex} \ldots <_{lex} F_n$
and the root of each $E_i$ is below the root of each $F_j$.
In that case, we may also write $(x,y,\mathcal{E}, \mathcal{F})$ for 
$\{z\in T_K\colon x< \meet(y,z)< y\}$.

We say that $(A,p) \in \SB_1$ is in a {\em simple normal form} if
(1) there are orbits 
$P=\{a_0,\ldots, a_n\}<_{lex} Q=\{b_0,\ldots, b_n\}$, $n\geq 2$,
of $p$, such that
 for every $i=0,\ldots,n-1$, $\meet(a_i,a_{i+1})<\meet(b_i,b_{i+1})$
and for every $i=0,\ldots,n-2$, $\meet(b_i,b_{i+1})<\meet(a_{i+1},a_{i+2})$, 
(2) any non-trivial orbit $O=\{c_0,\ldots,c_l\}$ in $A$ is $\leq_{lex}$-contained in
$P$ or in $Q$ and it holds $l=n-1$, 
(3)  for any $x$ with $p(x)=x$ it holds $\max\{a_0,a_n\}<_{lex} x<_{lex} \min\{b_0,b_n\}$,
where $\min$ and $\max$ are taken with respect to the $<_{lex}$ order.

We say that $(A,p)$ is in a {\em  normal form} if
(1) $A=\zdef(p)$, and any non-constant orbit of $p$ has at least 3 elements,
(2) $p$ cannot be extended
to a partial automorphism $q$ such that some two orbits that did not intertwine (or meet-intertwine) in 
$p$ now they intertwine (or meet-intertwine, respectively) or they form  one orbit
in $q$, and (3) there is a partition $\mathcal{P}^A$ of $A$ into 
singletons $\{x\}$ and
closed $\leq_{lex}$-intervals that will be grouped into pairs  $([a,b], [c,d])$ so that
if $\{x\}\in\mathcal{P}^A$, then $p(x)=x$, and if $([a,b], [c,d])\in\mathcal{P}^A$, 
then the structure $p\restriction ([a,b]\cup [c,d])$ is in a simple normal form without a non-trivial orbit, witnessed by some
$P=\{a_0,\ldots, a_n\}$
and $Q=\{b_0,\ldots, b_n\}$ with $a=a_0$, $b=a_n$, $c=b_n$, $d=b_0$ (or $b=a_0$, $a=a_n$, $d=b_n$, $c=b_0$).
We will sometimes identify $([a,b], [c,d])\in\mathcal{P}^A$ with the set $[a,b]\cup [c,d]$.

\begin{lemma}
Any $(A,p)\in\SB_1$ can be extended to $(A',p')\in\SB_1$, which is in a normal form.
\end{lemma}
\begin{proof}
Conditions (1) and (2) can be easily satisfied. To have (3), 
consider the equivalence relation on orbits: $O$ and $P$ are equivalent iff there is a sequence of orbits
$O=O_1,\ldots, P=O_n$ such that for each $i$, $O_i$ and $O_{i+1}$ intertwine or meet intertwine.
An equivalence class $E$ either  is a singleton  containing a  constant orbit, or it does not contain a constant orbit.
In the second case, after extending $A$ if necessary, the class $E$ contains two meet-intertwining orbits 
$P_0 <_{lex}Q_0$ (there are usually many such choices).
Extend $P_0$ to $P$ and $Q_0$ to $Q$ so that every orbit in $E$ is contained in $P$
or in $Q$. Finally, extend the remaining orbits in $E$ so that (2) in the definition of simple normal form is satisfied.
We do the induction on the number of equivalence classes $E$.
\end{proof}

Let $(A,p) \in\SB_1$ be in a normal form 
and let $X\in\bP^A$. Then $(X,p_X) \in \SB_1$, where $p_X=p\restriction X$, is in a simple normal form.
The tree $T_X$ can be naturally identified with a 
subtree of $T_{A}$ (in fact, $T_X$ is the  closure of $X$ in $T_A$ under taking the meet),
let $\rho_X$ be the root of $T_X$, and let $\Cone_X=\Cone_{\rho_X}$.
To $X$ as above we  associate $X^* \in\SB$, and $(X^*,p_X^*) \in \SB_1$ in a simple normal form as follows.
If $X=\{x\}$, let $p^*(x)=x$. Otherwise, if $X=([a,b],[c,d])$, we let 
\[X^*=X\cup \{ z=\Cone_{Y}\colon z \text{ is  maximal }  \text{ in }
(\{ \Cone_{Z}\subsetneq \Cone_{X} \colon Z\in\bP^A\},\subseteq)\  \}
\]
%
and we let $p_X^*$ to be the extension of $p_X$ that is equal to the identity on $X^*\setminus X$.
The set of all $X^*$ obtained in this way we denote by $(\bP^A)^*$.
Definitions introduced above are illustrated in Example \ref{exdef}.

\begin{proof}[Proof of Theorem \ref{boron}]

	Let $(A,p)\in\SB_1$ and consider extensions $(B,q), (C,r) \in\SB_1$ of $(A,p)$. By $\phi$, $\psi$, we denote the identity embeddings of $(A,p)$ into $(B,q)$, $(C,r)$, respectively. We show that if $(A,p)$ is in a normal form, then we can  amalgamate $(B,q)$ and $(C,r)$ over $(A,p)$. Since the family of all elements in a normal form is cofinal in $\SB_1$, this will finish the proof.

	{\bf{(1)}} Suppose that $(A,p)$ is in a simple normal form.
	
	Let $P=\{a_0,\ldots, a_n\}$ and $Q=\{b_0,\ldots, b_n\}$ be as in the definition of the simple normal form. 
	Without loss of generality, $P$ is increasing, and hence $Q$ is decreasing.
	Set $t_i=\meet(a_i,a_{i+1})$ and let $s_i=\meet(b_i,b_{i+1})$. Note that all trees $T_{X_i}$ with
	$X_i=[a_i,a_{i+1}]_{lex}$ are isomorphic, and all trees $T_{Y_i}$ with
	$Y_i=[b_{i+1}, b_i]_{lex}$ are isomorphic as well.

	Pick some $N$  such that each of $q^{-N}(a_0)$, $q^{N}(a_n)$, $q^{-N}(b_0)$, $q^{N}(b_n)$,
	$r^{-N}(a_0)$, $r^{N}(a_n)$, $r^{-N}(b_0)$, $r^{N}(b_n)$ is undefined.
	Let $\overline{\phi}\colon T_A\to T_B$ and $\overline{\psi}\colon T_A\to T_C$ be the unique 
	meet-preserving extensions of $\phi$ and $\psi$.
	
	For every $k$, let $(A_k,p_k)$ be defined as follows. Take an extension $p'$ of $p$ such that
	for every $x\in[a_0,a_1)_{lex}\cup (b_1,b_0]_{lex}$, the values $q^{n+k}(x)$ and $q^{-k}(x)$ are defined
	and every orbit in $p'$ extends an orbit in $p$. Then let
	\[ p_k=p'\restriction [a_{-k}, a_{n+k}]_{lex}\cup\{(c,c)\colon  c \in A, \ a_n<_{lex} c <_{lex} b_n\}\cup p'\restriction [b_{-k},b_{n+k}]_{lex},\]
	where $a_i=q^i(a_0)$ and  $b_i=q^i(b_0)$, and let $A_k=\zdef(p_k)$.  
	Note that $p= p_0$ and that each $p_k$ 
	is in a simple normal form as witnessed by $P_k=\{a_{-k},\ldots, a_{n+k}\}$ and 
	$Q_k=\{b_{-k},\ldots, b_{n+k}\}$. 
	Consider  $D_0=A_N$, $s_0=p_{N}$ for $N$ as above.

	Let 
	\[ A_B=\{y\in B\colon (\exists x\in A, k\in\mathbb{Z})\ q^k(x)=y\}, \]
	and define $A_C$ similarly.
	Let $\alpha\colon A_B\to D_0$ and $\beta\colon A_C\to D_0$ be the unique embeddings that agree on $A$. 
	Denote by
	$\bar{\alpha}\colon T_{A_B}\to T_{D_0}$ and $\bar{\beta}\colon T_{A_C}\to T_{D_0}$
	the tree embeddings corresponding to $\alpha$ and $\beta$.
	We clearly have $\rho_B\leq\overline{\phi}(\rho_A)$ or $\rho_C\leq\overline{\psi}(\rho_A)$,
	where $\rho_A, \rho_B, \rho_C$ are roots of $T_A, T_B, T_C$, and both inequalities can be strict.
	Then, to obtain the required amalgam $(D,s)$, first consider
	$\widehat{T}_{D_0}$  obtained from $T_{D_0}$ by adding  a new point $v$, which
	satisfies $v< \rho_{D_0}$, where $\rho_{D_0}$ is the root of $T_{D_0}$.
	Next, fix a segment $(x,y)$  in $\widehat{T}_{D_0}$, and let
	$(x^B, y^B)=(\bar{\alpha}^{-1}(x),\bar{\alpha}^{-1}(y))$, if defined,  
	and $(x^C,y^C)=(\bar{\beta}^{-1}(x),\bar{\beta}^{-1}(y))$, if defined.
	Suppose that  $\mathcal{E}^B$, $\mathcal{E}^C$,
	$\mathcal{F}^B$, $\mathcal{F}^C$ are such that 
	$\{z\in T_B\colon x^B<\meet(y^B,z)< y^B\}$ can be identified with 
	$(x^B,y^B, \mathcal{E}^B,\mathcal{F}^B)$ and 
	$\{z\in T_C\colon x^C<\meet(y^C,z)< y^C\}$ can be identified with 
	$(x^C,y^C,\mathcal{E}^C,\mathcal{F}^C)$.
	Then replace $(x,y)$ in $\widehat{T}_{D_0}$
	by $(x,y,(\mathcal{E}^B\mathcal{E}^C), (\mathcal{F}^C\mathcal{F}^B))$, where 
	$(\mathcal{E}^B\mathcal{E}^C)$ and  $(\mathcal{F}^C\mathcal{F}^B)$ are concatenations of sequences
	$\mathcal{E}^B$ with $\mathcal{E}^C$ and $\mathcal{F}^C$ with $\mathcal{F}^B$, respectively.
	We additionally require that the root of every $T_{E^C_i}$ is above the root of every $T_{F^B_j}$ and
	the root of every $T_{F^C_i}$ is above the root of every $T_{E^B_j}$, where 
	$E^C_i\in \mathcal{E}^C$, etc.
	The obtained tree $T$ defines
	$D\in\SB$ such that $T_D=T$, 
	and it defines embeddings $\alpha\colon B\to D$ and $\beta\colon C\to D$ of structures in $\SB$.
	We let $s(a)=b$ iff $\alpha^{-1}(a),\alpha^{-1}(b)$ are defined and $q(\alpha^{-1}(a))=\alpha^{-1}(b)$,  or
	$\beta^{-1}(a),\beta^{-1}(b)$ are defined and $r(\beta^{-1}(a))=\beta^{-1}(b)$, or
	$p^N(a)=b$.
	For every segment $(x,y)$ in $\widehat{T}_{D_0}$ and $z\in T_D$, $x<z< y$, we will call the 
	subtree $\Cone_z$ of $T_D$ a {\em triangle} (coming from $B$ or from $C$).

	We have to show that $s$ is a partial automorphism of $D$. Clearly $<_{lex}$ is preserved.
	Let $\bar{s}_0$ be the meet-preserving extension
	of $s_0$ to $\widehat{T}_{D_0}$. Then, clearly, a triangle attached to a segment $(x,y)$ is mapped
	to a triangle attached to a segment $(\bar{s}_0(x),\bar{s}_0(y))$. The key observation is 
	that for every $k$, and in particular for $k=N$, and every $X\subseteq \dom(p_k)$, the tree 
	$T_X$ are isomorphic with the tree  $T_{s_0(X)}$ via the
	tree isomorphism extending the bijection $x\mapsto  s_0(x)$, $x\in X$. 
	Moreover, for $a,b\in D$ that lie in different triangles, $\meet(a,b)$ is equal to the meet of the roots of the
	triangles to which $a$ and $b$ belong. Note also that if $\rho \leq \rho'$ are roots of two triangles 
	then $\bar{s}(\rho)\leq \bar{s}(\rho')$.
	Therefore, if $x,y,z\in D$ lie in different triangles, we have $S(x,y,z)$ iff $S(s(x),s(y),s(z))$.
	Clearly, if all $x,y,z$ lie in the same triangle, then the conclusion holds. 
	If $x\leq_{lex} y \leq_{lex} z$ and $x$ and $y$ lie in a triangle $T$ and $z$ lies in a different triangle $S$,
	if $\rho_T$ and $\rho_S$ denote roots of $S$ and $T$, then $S(x,y,z)$ holds iff $\rho_S< \rho_T$ and
	$S(s(x),s(y),s(z))$ holds iff $\bar{s}_0(\rho_S)< \bar{s}_0(\rho_T)$, hence 
	$S(x,y,z)$ iff $S(s(x),s(y),s(z))$.
	The case when $y$ and $z$ lies in the same triangle and $x$ lies in a different one is analogous.

	{\bf{(2)}} Suppose that $(A,p)$ is in  a normal form.
	
	Without loss of generality, $(B,q)$ and $(C,r)$ are also in the normal form. 
	Let $\widehat{\bP}^B$ be the partition of $B$ into points and pairs of closed $\leq_{lex}$-intervals, which is a coarsening of $\bP^B$,
	and has the following property: 
	for every $([a,b],[c,d])\in\widehat{\bP}^B$ there is exactly  one  $([a_0,b_0],[c_0,d_0])\in\bP^A$ 
	such that $[a_0,b_0]\subseteq [a,b]$ and $[c_0,d_0]\subseteq [c,d]$.
	We define $(\bP^A)^*$ out of $\bP^A$, and the corresponding partial automorphisms $p_X^*$, in a way explained earlier. We define $(\widehat{\bP}^B)^*$ and $(\widehat{\bP}^C)^*$ similarly, but with respect to 
	 $\widehat{\bP}^B$ and  $\widehat{\bP}^C$.  Let
	\[X^*=X\cup \{ z=\Cone_{Y}\colon z \text{ is  maximal }  \text{ in }
(\{ \Cone_{Z}\subsetneq \Cone_{X} \colon Z\in \widehat{\bP}^B\},\subseteq)\  \},
\] and we let $p_X^*$ to be the extension of $p_X$ that is equal to the identity on $X^*\setminus X$.
The set of all $X^*$ obtained in this way we denote by $(\widehat{\bP}^B)^*$. We similarly define
	$(\widehat{\bP}^C)^*$.
	For a given  $X \in\bP^A$, let
	$X_B \in \widehat{\bP}^B$ and
	$X_C \in \widehat{\bP}^C$ be the unique sets containing $X$.
	Amalgamate $(X^*_B, q^*_{X_B})$ and $(X^*_C, r^*_{X_C})$ over $(X^*, p^*_{X})$, as we did in (1). We obtain $D_X\in\SB$, a partial automorphism $s_X$ of $D_X$, and a pair of embeddings 
	$\alpha_X\colon X^*_B\to (D_X,s_X)$ and $\beta_X\colon X^*_C\to (D_X,s_X)$
	 such that $\alpha_X \restriction X^*=\beta_X \restriction X^*$.  
	Let $\rho_X$ be the root of $T_X$, and 
	let $\rho_1^X, \ldots, \rho_{l_X}^X$ be an enumeration of all  $z=\Cone_Y$
	from the definition of $X^*$. To the tuple $(X^*, \rho_X, \rho_1^X, \ldots, \rho_{l_X}^X)$ we associate the tuple
	$(D_X, \rho_{D_X}, \rho_1^{D_X}, \ldots, \rho_{l_X}^{D_X})$, where $\rho_{D_X}$ is the root of $T_{D_X}$, and
	$\rho_i^{D_X}=\alpha_X(\rho_i^X)=\beta_X(\rho_i^X)$ for each $i$.
	By the definition of $(\bP^A)^*$, the tree $T_A$ is the disjoint union of 
	$\{T_{X^*}\colon X\in \bP^A\}$ after the identification of each $\rho_i^{X}$ 
	with a $\rho_{Y}$ for some $Y\in (\bP^A)$. 
	Finally, let $T_D$ be obtained by taking the disjoint union of 
	$\{T_{D_{X}}\colon X\in \bP^A\}$, and then identifying a $\rho_i^{D_{X}}$ with a $\rho_{D_{Y}}$ if and only if $\rho_i^{X}$ was identified with a $\rho_{Y}$, $X,Y\in \bP^A$. 
	This $T_D$  defines $D\in\SB$ 
	and let $s=\bigcup_{X\in\bP^A} s_X\restriction X$. Then $(D,s)\in \SB_1$. Clearly, $s$ preserves $<_{lex}$. 
	To show that it preserves $S$, observe first the following. 
	
	(i) If $x\in D_X$ and $y\in D_Y$ with $X\neq Y$, and $\rho_{D_X}$ and $\rho_{D_Y}$ are incomparable
	in $(T_D, \leq)$, then $\meet(x,y)=\meet(\rho_{D_X},\rho_{D_Y})=\meet(s(x), s(y))$.
	
	(ii)  If $x\in D_X$ and $y\in D_Y$ with $X\neq Y$ and $\rho_{D_X}\leq  \rho_{D_Y}$, then
	$\meet(x,y)=\meet(x,\rho_{D_Y})$ and $\meet(s(x), s(y))=\meet(s(x), \rho_{D_Y})$.
	Moreover, in that case, if $y'\in D_Y$, then $\meet(x,y),\meet(x,y')< \meet(y,y')$.
	
	Let $(x,y,z)$ be a $<_{lex}$ ordered triple of points in $D$. There are several cases to consider. 
	Clearly, if there is some  $X\in \bP^A$ such that $x,y,z\in D_X$, then $S(x,y,x)$ iff $S(s(x),s(y),s(z))$.
	If $x,y\in D_X$, $z\in D_Y$ with $X\neq Y$  and  $\rho_{D_X}$ and $\rho_{D_Y}$ are incomparable,
	then  both $S(x,y,z)$ and
	$S(s(x),s(y),s(z))$ hold.
	Similarly, if $x\in D_X$ and $y,z\in D_Y$ with $X\neq Y$ and  $\rho_{D_X}$ and $\rho_{D_Y}$ are incomparable
	then none of $S(x,y,z)$,
	$S(s(x),s(y),s(z))$ holds.
	In the case when $\rho_{D_X}\leq \rho_{D_Y}$, two of the $x,y,z$ belong to $D_X$ and the remaining point belongs to $D_Y$, $X\neq Y$, then using (ii) we get that $s$ preserves $S$ on $(x,y,z)$ because
	$p_X$ preserves $S$ (the point belonging to $D_Y$ we can replace 
	with an appropriate $\rho^{D_X}_i$). If  two of the $x,y,z$ belong to $D_Y$ and the remaining point belongs to $D_X$, use the second sentence of (ii) to get the conclusion.
	Finally, suppose that $x\in D_X$, $y\in D_Y$, $z\in D_Z$ with $X,Y,Z$ pairwise different. There are a few cases
	to consider, $\rho_{D_X}$  and $\rho_{D_Y}$ can be  comparable or not, and the  same for the other
	two pairs. Each time, reasoning similarly as above and using (i) and (ii), we get the required conclusion.
	
\end{proof}

\begin{example}\label{exdef}
	{\rm{
			Let $(A,p)\in\SB_1$ be as in Figure \ref{fig:M1} with $p(a_0)=a_1$, $p(a_1)=a_2$, $p(a'_0)=a'_1$, $p(a'_1)=a'_2$, etc.
			Then $\mathcal{P}^A=\{X=([a_0,a_2], [a'_2,a'_0]), Y=([b_0,b_3], [b'_3,b'_0]), Z=([d_0,d_2], [d'_2,d'_0])\}$.
			Figure \ref{fig:M2} illustrates $X^*,Y^*, Z^*$, where $p^*_Z(x)=x$ and $p^*_Z(y)=y$.}}

	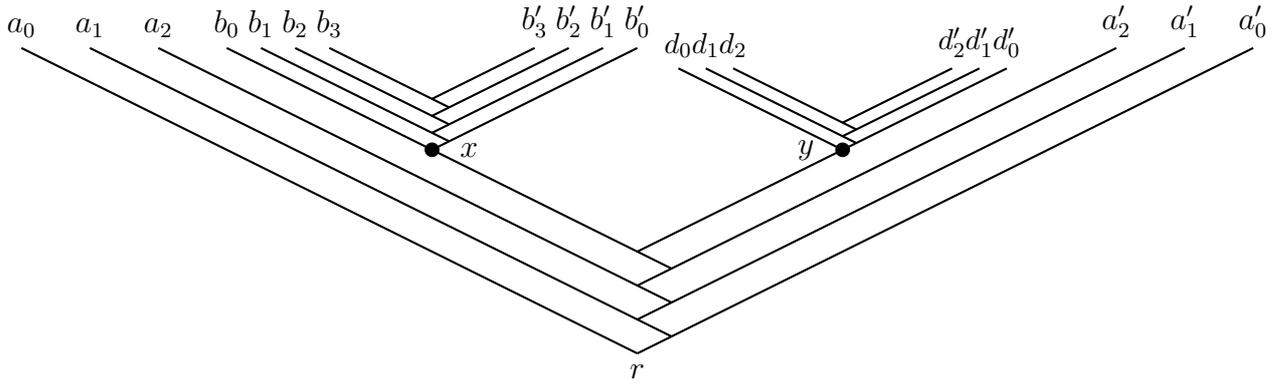
\begin{figure}[H]
		\begin{tikzpicture}[scale=0.9]
		\draw [thick] (0,0) -- (-9,4.5) ;
		\draw [thick] (.5,.25) -- (-8,4.5);
		\draw [thick] (.5,.75) -- (-7,4.5);
		\draw [thick] (0,0) -- (9,4.5);
		\draw [thick] (0,.5) -- (8,4.5);
		\draw [thick] (0,1) -- (7,4.5);
		\node[above] at (-9,4.5) {$a_0$}; 
		\node[above] at (-8,4.5) {$a_1$}; 
		\node[above] at (-7,4.5) {$a_2$};  
		\node[above] at (9,4.5) {$a'_0$}; 
		\node[above] at (8,4.5) {$a'_1$}; 
		\node[above] at (7,4.5) {$a'_2$}; 
		\draw [thick] (.5, 1.25) -- (-3,3);
		\draw [thick] (0,1.5) -- (3,3);
		\node[right] at (-2.75,3) {$x$};
		\node[left] at (2.75,3) {$y$};
		\node[below] at (0,0) {$r$};
		
		\draw [thick] (-3,3) -- (-6/2-3,3/2+3);
		\draw [thick] (.5/2-3,.25/2+3) -- (-5/2-3,3/2+3);
		\draw [thick] (.5/2-3,.75/2+3) -- (-4/2-3,3/2+3);
		\draw [thick] (.5/2-3,1.25/2+3) -- (-3/2-3,3/2+3);
		\draw [thick] (-3,3) -- (6/2-3,3/2+3);
		\draw [thick] (-3,.5/2+3) -- (5/2-3,3/2+3);
		\draw [thick] (-3,1/2+3) -- (4/2-3,3/2+3);
		\draw [thick] (-3,1.5/2+3) -- (3/2-3,3/2+3);
		\node[above] at (-6/2-3,3/2+3) {$b_0$}; 
		\node[above] at (-5/2-3,3/2+3) {$b_1$}; 
		\node[above] at (-4/2-3,3/2+3) {$b_2$}; 
		\node[above] at (-3/2-3,3/2+3) {$b_3$}; 
		\node[above] at (6/2-3,3/2+3) {$b'_0$}; 
		\node[above] at (5/2-3,3/2+3) {$b'_1$}; 
		\node[above] at  (4/2-3,3/2+3) {$b'_2$}; 
		\node[above] at (3/2-3,3/2+3) {$b'_3$}; 
		
		\draw [thick] (3,3) -- (-6/2.5+3,3/2.5+3);
		\draw [thick] (.5/2.5+3,.25/2.5+3) -- (-5/2.5+3,3/2.5+3);
		\draw [thick] (.5/2.5+3,.75/2.5+3) -- (-4/2.5+3,3/2.5+3);
		\draw [thick] (3,3) -- (6/2.5+3,3/2.5+3);
		\draw [thick] (3,.5/2.5+3) -- (5/2.5+3,3/2.5+3);
		\draw [thick] (3,1/2.5+3) -- (4/2.5+3,3/2.5+3);
		\node[above] at (-6/2.5+3,3/2.5+3) {\small{$d_0$}}; 
		\node[above] at  (-5/2.5+3,3/2.5+3) {\small{$d_1$}}; 
		\node[above] at (-4/2.5+3,3/2.5+3) {\small{$d_2$}}; 
		\node[above] at (6/2.5+3,3/2.5+3) {\small{$d'_0$}}; 
		\node[above] at (5/2.5+3,3/2.5+3) {\small{$d'_1$}}; 
		\node[above] at (4/2.5+3,3/2.5+3) {\small{$d'_2$}}; 
		
		\draw [fill] (-3,3) circle [radius=0.1];
		\draw [fill] (3,3) circle [radius=0.1];

		\end{tikzpicture}
		\caption{A structure $A$} \label{fig:M1}
	\end{figure}

	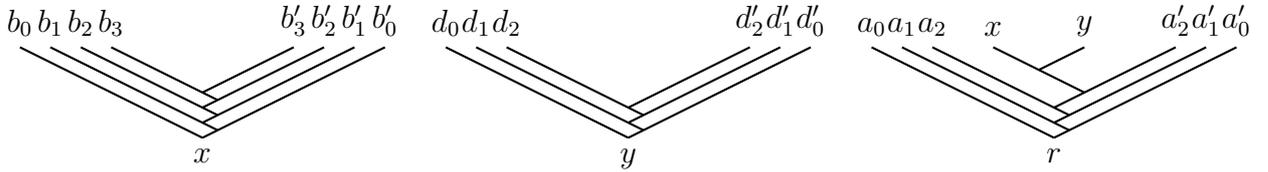
\begin{figure}[H]
		\begin{tikzpicture}[scale=.4]
		\draw [thick] (0,0) -- (-6,3);
		\draw [thick] (.5,.25) -- (-5,3);
		\draw [thick] (.5,.75) -- (-4,3);
		\draw [thick] (.5,1.25) -- (-3,3);
		\draw [thick] (0,0) -- (6,3);
		\draw [thick] (0,.5) -- (5,3);
		\draw [thick] (0,1) -- (4,3);
		\draw [thick] (0,1.5) -- (3,3);
		\node[above] at (-6,3) {$b_0$}; 
		\node[above] at (-5,3) {$b_1$}; 
		\node[above] at (-4,3) {$b_2$}; 
		\node[above] at (-3,3) {$b_3$}; 
		\node[above] at (6,3) {$b'_0$}; 
		\node[above] at (5,3) {$b'_1$}; 
		\node[above] at (4,3) {$b'_2$}; 
		\node[above] at (3,3) {$b'_3$}; 
		\node[below] at (0,0) {$x$}; 
		
		\begin{scope}[xshift=14cm]
		\draw [thick] (0,0) -- (-6,3);
		\draw [thick] (.5,.25) -- (-5,3);
		\draw [thick] (.5,.75) -- (-4,3);
		\draw [thick] (0,0) -- (6,3);
		\draw [thick] (0,.5) -- (5,3);
		\draw [thick] (0,1) -- (4,3);
		\node[above] at (-6,3) {$d_0$}; 
		\node[above] at (-5,3) {$d_1$}; 
		\node[above] at (-4,3) {$d_2$};  
		\node[above] at (6,3) {$d'_0$}; 
		\node[above] at (5,3) {$d'_1$}; 
		\node[above] at (4,3) {$d'_2$}; 
		\node[below] at (0,0) {$y$}; 
		\end{scope}
		
		\begin{scope}[xshift=28cm]
		\draw [thick] (0,0) -- (-6,3);
		\draw [thick] (.5,.25) -- (-5,3);
		\draw [thick] (.5,.75) -- (-4,3);
		\draw [thick] (0,0) -- (6,3);
		\draw [thick] (0,.5) -- (5,3);
		\draw [thick] (0,1) -- (4,3);
		\node[above] at (-6,3) {$a_0$}; 
		\node[above] at (-5,3) {$a_1$}; 
		\node[above] at (-4,3) {$a_2$};  
		\node[above] at (6,3) {$a'_0$}; 
		\node[above] at (5,3) {$a'_1$}; 
		\node[above] at (4,3) {$a'_2$}; 
		\draw [thick] (1, 1.5) -- (-2,3);
		\draw [thick]   (-.5,2.25) -- (1,3);
		\node[above] at (-2,3) {$x$};
		\node[above] at (1,3) {$y$};
		\node[below] at (0,0) {$r$};
		\end{scope}
		\end{tikzpicture}
		\caption{From left to right: $X^*, Y^*, Z^*$} \label{fig:M2}
	\end{figure}

\end{example}

\subsection{Ordered posets - a comeager conjugacy class}
In this part, we will show that the automorphism group of the universal  ordered poset
has a comeager conjugacy class.

A {\em poset} is a shortcut for a partially ordered set.
By an {\em ordered  poset},   we mean a structure of the form $(P,\prec, <)$, where $(P,\prec)$ is a finite poset, and $<$ is a linear ordering of $P$ extending $\prec$. We denote the class of all finite ordered posets by $\p$. Then the  {\em universal ordered poset} is the \fra limit of $\p$. 
Kuske-Truss \cite{KT} showed that the class of partial automorphisms of finite posets has CAP, and hence  the corresponding automorphism group has a comeager conjugacy class. The same turns out to be true for the class $\p_1$.

We will see that the proof of Kuske-Truss generalizes to our context.
\begin{theorem}\label{wapposet}
The class $\p_1$ has CAP. 
\end{theorem}

Below, we will always use the symbol $\prec$ for the poset relation, and $<$ for the linear order. For $A=(A,p)\in\p_1$, and an orbit $O=\{a_0, a_1, \ldots, a_{n}\}$ of $p$, we say that $O$ is \emph{$<$-increasing} if $a_0<\ldots< a_n$, otherwise, it is \emph{$<$-decreasing}.
As with boron trees, if $(B,q)\in\p_1$ extends $(A,p)$,  the orbit of $q$ extending $O$ will be denoted by $O_B$.

We say that a pair of orbits $(O,N)$ is {\em determined} if for any extensions $(B,q),(C,r)$ of $(A,p)$, such that for every $k\in\mathbb{Z}$ and $x \in O \cup N$ we have that $q^k(x)$ is defined iff $r^k(x)$ is defined, the following holds: $(O_B\cup N_B, q \restriction(O_B\cup N_B))$ and
$(O_C\cup N_C,r \restriction(O_C\cup N_C))$ are isomorphic via a mapping which is the identity on $O\cup N$.
Clearly, $(O,N)$ is determined iff $(N,O)$ is determined.
An orbit $O$ is {\em determined} if the pair $(O,O)$ is determined.

Let $(A,p)$ be given.
Let $O$ be an orbit in $(A,p)$ and $x\in O$. Denote
\[ t(O,x)=\{n\in\ZZ\colon p^n(x)\in O \text{ and } x\prec p^{n}(x)\}.\]
Note that if $O$ is $<$-increasing then $t(O,x)$ consists of non-negative integers and if $O$ is $<$-decreasing
then $t(O,x)$ consists of non-positive integers.
Moreover, if $n\in t(O,x)$ then by the transitivity of $\prec$, for every positive integer $k$, we have that
if $p^{kn}(x)\in O$ then $kn\in t(O,x)$.
More generally, if $(O,N)$ are orbits in $(A,p)$ and $x\in O$, $y\in N$, then let
\[t(O,N,x,y)=\{n\in\ZZ\colon p^n(y)\in N \text{ and } x\prec p^{n}(y)\}.\]
We say that $(O,N,x,y)$ is {\em positive determined} if 
orbits $O$ and $N$ are determined and
for any extensions $(B,q),(C,r)$ of $(A,p)$, such that for every $k\in\mathbb{Z}$  we have that $q^k(x)$ is defined iff $r^k(x)$ is defined and $q^k(y)$ is defined iff $r^k(y)$ is defined, we have that $t(O_B,N_B,x,y)\cap\NN=t(O_C,N_C,x,y)\cap\NN$. We similarly define when 
$(O,N,x,y)$ is {\em negative determined}. We let $(O,N,x,y)$ to be {\em determined}
 iff  
$(O,N,x,y)$ is both positive and negative determined.
The  $(O,x)$ is {\em determined} if for any extensions $(B,q),(C,r)$ of $(A,p)$, such that for every $k\in\mathbb{Z}$  we have that $q^k(x)$ is defined iff $r^k(x)$ is defined, it holds  $t(O_B,x)=t(O_C,x)$. 
Note that a pair of orbits $(O,N)$ is determined iff for every $x\in O$, $y\in N$, the $(O,N,x,y)$ and $(N,O,y,x)$
 are determined iff
for some $x\in O$, $y\in N$, the  $(O,N,x,y)$ and $(N,O,y,x)$ are
 determined. Similarly, $O$ is determined iff for some/every
$x\in O$, the  $(O,x)$ is determined.

An orbit $O$ will be called an {\em antichain} if for every extension $(B,q)$ of $(A,p)$ and for every/some $x\in O$, we have $t(O_B,x)=\emptyset$. 
An $X\subset \ZZ$ is called {\em positive eventually periodic} if there exist $N,k\geq 0$ such that
$X\cap[N,\infty)=\{N+kn\colon n\geq 0\}$. 
The number $k$ we will call the {\em positive period} of $X$.
We similarly define a {\em negative eventually periodic} set and the {\em negative period}.
We call a set {\em eventually periodic} if it is both positive and negative periodic.
We will call two orbits $O=\{a_0,\ldots,a_m\}$ and $N=\{b_0,\ldots, b_n\}$ of $p$ 
{\em intertwining} if the $<$-intervals $(\min\{a_0,a_m\}, \max\{a_0,a_m\})_<$ and 
 $(\min\{b_0,b_n\}, \max\{b_0,b_n\})_<$ intersect.

\begin{proof}[Proof of Theorem \ref{wapposet}]
The proof is  similar to the proof of Kuske-Truss \cite{KT}. 

{\bf{Step 1.}} Every $(A,p) \in \p_1$ can be extended to some $(B,q)\in\p_1$, in which every orbit is determined.
Moreover, we can do it so that we do not add new orbits.

Take an orbit $O$ and $x\in O$, and without loss of generality suppose that $O$ is $<$-increasing. If $O$ is an antichain
then it is already determined. Otherwise, let (perhaps after passing to an extension) $k\in t(O,x)$. Now  for every positive integer $n$, we have that
if $p^{kn}(x)\in O$ then $kn\in t(O,x)$. Using this remark, we obtain  $(B,q)$, an extension  of $(A,p)$, such that for every  $i=1,\ldots,k-1$,
if there is an extension $(B_1,q_1)$ of $(B,q)$ with $n_ik+i\in t(O_{B_1},x)$, for some $n_i\geq 0$, then, in fact,
for every extension $(B_2,q_2)$ of $(B,q)$, it holds $n_ik+i\in t(O_{B_2},x)$, as long as $n_ik+i\in O_{B_2}$.
To obtain such a $(B,q)$, we construct a sequence of extensions $(A_1,p_1),\ldots, (A_{k-1}, p_{k-1})$ of $(A,p)$ 
such that $(A_i,p_i)$ has the required property for $i$, and we take $(B,q)=(A_{k-1}, p_{k-1})$.
Then clearly $O_B$ is determined.

{\bf{Step 2.}} Every $(B,q) \in \p_1$ can be extended to some $(D,s)\in\p_1$, in which every pair of orbits is determined.
For this we find a $(D,s)$ such that for every pair of orbits $(O,N)$, $x\in O$, $y\in N$, there is an almost periodic
set $X\subseteq\ZZ$ such that  for every extension $(D_1,s_1)$ of $(D,s)$, if $m_1,m_2$ are the least such that 
${s_1}^{m_1}(y)$ and ${s_1}^{-m_2}(y)$ are undefined,  we have $X\cap(-m_1,m_2)=t(O_{D},x)$.

{\bf{Step 2a.}} The $(B,q) $ obtained in Step 1 can be extended to some $(C,r)\in\p_1$ such that all pairs of orbits in which both orbits are antichains, are determined. Moreover, we can do it in a way that for each such a pair we add  four new orbits,
none of which is an antichain.

We fix a pair $(O,N)$ of such orbits and let $x\in O$.
As in Kuske-Truss, after possibly extending $N$, find $y\in N$ and $0<n$ with $q^n(y)\in\rng(q)\setminus\dom(q)$
such that for every 
$z\in\zdef(q)\setminus N$, it holds: $z\prec y$ iff $z\prec q^n(y)$, $y\prec z$ iff $q^n(y)\prec z$,
$z$ is incomparable with $y$ iff $z$ is incomparable with $q^n(y)$.
For this pick some $y_0\in N$ and, possibly extending $q$, using the pigeonhole principle,  choose $k_1<k_2$ sufficiently large so that $q^{k_2}(y_0)\in\rng(q)\setminus\dom(q)$,
 $z\prec q^{k_1}(y_0)$ iff $z\prec q^{k_2}(y_0)$, $q^{k_1}(y_0)\prec z$ iff $q^{k_2}(y_0)\prec z$,
$z$ is incomparable with $q^{k_1}(y_0)$ iff $z$ is incomparable with $q^{k_2}(y_0)$.
Take $y=q^{k_1}(y_0)$ and $n=k_2-k_1$.

 Let us proceed to the construction. Take $\{a_i\colon  0\leq i\leq n\}$  and $\{b_i\colon  0\leq i\leq n\}$
 disjoint from each other and from
$\zdef(q)$ and such that for all $0\leq i<j\leq n$, $z\in\zdef(q)\setminus N$ :
 \begin{enumerate}
 \item $a_i$ and $a_j$ are incomparable except for $a_0\prec a_n$,
\item $b_i$ and $b_j$ are incomparable except for $b_n\prec b_0$,
\item $a_i\prec q^i(y)\prec b_i$,
\item $z\prec b_i$ iff $z\prec q^i(y)$,
\item $a_i\prec z$ iff $q^i(y)\prec z$,
\item $a_i\prec b_i$, $a_0\prec b_n$, $a_n\prec  b_0$,
\item $a_0<\ldots<a_n<\zdef(q)<b_n<\ldots<b_0$.
\end{enumerate}
Denote the obtained structure  by $(B_1,\prec, <)$. Let $q_1$ extend $q$ by $q_1(a_i)=a_{i+1}$ and
$q_1(b_{i+1})=b_i$, $i=0,\ldots, n-1$.
It is straightforward to see that $\prec$ is transitive, $<$ extends $\prec$,  $p$ preserves $\prec$ and $<$.
As in Kuske-Truss, we have that in any extension $(B_2,q_2)$ of $(B_1,q_1)$,
\[t(O_{B_2},N_{B_2},x,y)=\{kn+i\colon k\in\NN, 0\leq i<n, x\prec q^i(y)\}\cap B_2,\]  hence $(O,N,x,y)$ is positive determined.
Indeed, let  $(B_2,q_2)$ be an extension of $(B_1,q_1)$ and denote 
$X=\{kn+i\colon k\in\NN, 0\leq i<n, x\prec q^i(y)\}\cap B_2$.
Since $a_0\prec a_n$ and $a_i\prec q^i(y)$, we get $t(O_{B_2},N_{B_2},x,y)\supseteq X$. 
Since $b_n\prec b_0$ and $q^i(y)\prec b_i$, using (4), we get $t(O_{B_2},N_{B_2},x,y)\subseteq X$.

 By considering $(B_1, q_1^{-1})$ and repeating the argument above, we further extend $(B_1,q_1) $ to 
 obtain $(C,r)$ in which  $(O,N)$ is determined.

{\bf{Step 2b.}}  The $(C,r) $ obtained in Step 2a can be extended to some $(D,s)\in\p_1$ such that all pairs of orbits such that at least one of them is not an antichain are determined. 
Moreover, we can do it in a way that we do not add new orbits.

We fix a pair $(O,N)$ of such orbits and let $x\in O$ and $y\in N$.
  If for every extension $(C_1,r_1)$ of $(C,r)$, $t(O,N,x,y)=\emptyset$ and  $t(N,O,y,x)=\emptyset$,
   then  $(O,N)$ is already determined. Therefore, 
 by passing to an extension if necessary,  we assume that $t(O,N,x,y)\neq\emptyset$ or $t(N,O,y,x)\neq\emptyset$.
  Note that at least one of the sets $t(O,N,x,y)$ and $t(N,O,y,x)$ is empty. Without loss of generality,
  let us assume that $x\prec y$ and $t(N,O,y,x)=\emptyset$.

If $k\in t(O,x)$ and $l\in   t(N,y)$, then for every $n_1,n_2\geq 0$, if $q^{kn_1+ln_2}(y)\in N$, then
$kn_1+ln_2\in t(O,N,x,y)$. Hence if there is  a positive number $k\in t(O,x)$ or a positive number $l\in   t(N,y)$,
 reasoning as in Step 1, we can find an extension in which $(O,N,x,y)$ is positive determined. Similarly, if 
  there is  a negative number $k\in t(O,x)$ or a negative number $l\in   t(N,y)$,
 we can find an extension in which $(O,N,x,y)$ is negative determined. 
  Therefore, without loss of generality, what is left to be shown is the following. Suppose that $O$ and $N$ are
  $<$-increasing,  at least one of them is not an antichain, and $(O,N,x,y)$ is positive determined.
  Then we can extend $(C,r)$ to $(C_1,r_1)$ so that $(O_{C_1},N_{C_1},x,y)$ is negative determined.

Without loss of generality, $O$ is not an antichain. Take some $k$ such that $x\prec r^k(x)$. Then, clearly, for every $n>0$, it holds
$x\prec r^{nk}(x)$. Take an extension $(C_1,r_1)$ such that for every $0\leq i<k$ either for some $n_i$, we have
${r_1}^{-(n_ik+i)}(y)\in C_1$ and $x\prec {r_1}^{-(n_ik+i)}(y)$  does not hold, or, for every extension $(C_2,r_2)$ of 
$(C_1,r_1)$ 
and every $n>0$, $x\prec {r_1}^{-(nk+i)}(y)$  holds .
Then $(O_{C_1}, N_{C_1},x,y)$ is negative determined. Indeed, note that if $i$ and $n_0$ are such that 
$x\prec {r_1}^{-(n_0k+i)}(y)$   does not hold, then in every extension $(C_2,r_2)$ of $(C_1,r_1)$ and $n_1>n_0$,
$x\prec {r_1}^{-(n_1k+i)}(y)$,   does not hold either by the choice of $k$.

We apply this procedure to every pair of orbits such that at least one of them is not an antichain. 
The resulting extension, which we denote by $(D,s)$, is as required.

	\medskip

	We will show that $\p_1$ has the CAP,  i.e. we will show that for every
	 $(A_0,p_0)\in\p_1$
	there is $(A,p)\in\p_1$ extending $(A_0,p_0)$ such that for any $(B,q),(C,r)\in \p_1$ extending $(A,p)$
	there exists $(D,s)\in\p_1$, which is an amalgam of $(B,q)$ and $(C,r)$ over $(A,p)$.
	For this fix $(A_0,p_0)\in \p_1$ and extend it to an
	$(A,p)\in \p_1$ such that any pair of orbits in $A$ is determined, and there is no extension $(A_1,p_1)$ of $(A,p)$ in which some two orbits in $A$ that did not intertwine, become one orbit or they intertwine in $A_1$.
	Fix $(B,q),(C,r)\in \p_1$ extending $(A,p)$. 
	Without loss of generality, we have that $A\subseteq\dom(q)\cap\rng(q)$, and similarly,
	$A \subseteq\dom(r)\cap\rng(r)$, as well as that for every $a\in A$ and $n\in\mathbb{Z}$,
	$q^n(a)$ is defined iff $r^n(a)$ is defined.
	
	Enumerate the set $\{b\in B\colon (\exists a\in A, n\in\mathbb{Z}) \ b=r^n(a)\}$ into
	a $<$-increasing sequence 
	$a^B_1,a^B_2,\ldots, a^B_k$, and similarly,
	enumerate the set $\{c\in C\colon (\exists a\in A, n\in\mathbb{Z}) \ c=r^n(a)\}$ into
	$a^C_1,a^C_2,\ldots, a^C_k$ so that it is $<$-increasing.
	By the assumptions on $A$, for any $a,b\in A$ and $m,n\in\mathbb{Z} $,
	if $x_B=q^m(a)$, $y_B=q^n(b)$, $x_C=r^m(a)$, $y_C=q^n(b)$
	are defined, then $x_B< y_B$ iff $x_C< y_C$.
	Denote $B_i=\{b\in B\colon a_i^B< b< a^B_{i+1}\}$, $i=1,\ldots, k-1$, $B_0=\{b\in B\colon b< a^B_1\}$,
	$B_k=\{b\in B\colon a_k^B< b\}$, and similarly define $C_i$'s.
	
	We first amalgamate $B$ and $C$ over $A$ in $\p$. Let $D$ be the disjoint union
	of $B$ and $C$ with  $a_i^B$ and $a_i^C$ identified. 
	Set $a_i=a_i^B=a_i^C$, and let 
	\[B_0^D< C_0^D< a_1< B_1^D< C_1^D<a_2<\ldots< a_k< B_k^D< C_k^D.\]
	Denote by $\prec^B$ the partial ordering relation on $B$, by $\prec^C$ the partial ordering relation on $C$, and let $\prec^D$ be the transitive closure of $\prec^B \cup \prec^C$.
	Then $(D,\prec^D)$ is a partial ordering such that the linear ordering $<^D$ extends $\prec^D$, and so $D$ is an amalgam of $B$ and $C$ over~$A$.
	
	We finally let $s(x)=y$ iff  $x=b_1, y=b_2$ for some $b_1,b_2\in B$ and $q(b_1)=b_2$,
	or $x=c_1, y=c_2$ for some $c_1,c_2\in C$ and $r(c_1)=c_2$.  This is a partial automorphism.
	In particular, if for some $b \in B$,  $c \in C$, it holds $b\prec^D c$ and $s(b), s(c)$ are defined,
	then there is $a\in \{a_1,\ldots, a_k\}$ such that $b\prec^B a$ and $a \prec^C c$, or
	$b\prec^C a$ and $a \prec^B c$. 
	Without loss of generality, it holds
	$b\prec^B a$ and $a \prec^C c$. Note that 
	$q(a)$ and $r(a)$ are defined and hence we have 
	$q(b)\prec^B q(a)$ and $r(a)\prec^C r(c)$. This implies $s(b)\prec^D s(c)$.
	Hence $(D,s)$ is the required amalgam.
	
\end{proof}

\section{The two-dimensional case. Conjugacy classes.}

We  provide a condition, which we will use to obtain many examples of ordered
structures $M$ such that ${\rm Aut}(M)$ has no comeager $2$-dimensional diagonal conjugacy class.

Given a partial automorphism $p$ of a structure $A$, and  $a\in A$, say that   $a\in A$ is {\em  locked} by $p$ if
there are  $x\leq a\leq y$, $x,y\in A$ such that $p(x)=y$ or $p(y)=x$.

\begin{proposition}\label{nowap}
	Let $\f$ be an \fra order class.
	Suppose that for every $(A,p)\in\f_1$  and $a\in A$ not locked by $p$,
	there are extensions $(B,r), (C,s)\in\f_1$ of $(A,p)$ such that $r(a)< a$ and $a< s(a)$.
	Then  $\f_2$ has no WAP.
\end{proposition}
\begin{proof}
	It suffices to show that for a given 
	$(A',p',q')\in\f_2$ and $x\in A$ such that $x< p'(x)$ 
	there exists $(A'', p'',q'')\in\f_2$ which is an extension of $(A',p',q')$, and a word $w(s,t)\in F_2$, such that $w(p'',q'')(x)$ is defined and not locked by $p''$ or $q''$. 
	Indeed, let $(A, p, q)\in\f_2$ and $x\in A$ be such that $x< p(x)$ (wlog there is such an $x$ in $A$), let
	$(A', p',q')\in\f_2$ be an arbitrary extension of $(A,p,q)$, and let $(A'', p'',q'')\in\f_2$ and $w(s,t)$ be as above.
	Then $y=w(p'',q'')(x)$ is not locked (wlog) by $p''$.
	Apply the assumptions of the proposition to $(A'',p'')$ and $y$ to find the corresponding $(B,r), (C,s)\in\f_1$.
	Then $(B, p'',r), (C, p'', s)$ cannot be amalgamated over~$(A,p,q)$.

	We construct the required $(A'',p'',q'')$ and $w$ inductively.
	Let $(A''_0,p''_0,q''_0)=(A',p',q')$, and $w_0=1$. Suppose that we already constructed $(A''_n,p''_n,q''_n)$ and $w_n(s,t)$, and 
	suppose  that $w_{n-1}(p''_{n-1},q''_{n-1})(x)=w_{n-1}(p''_{n},q''_{n})(x)< w_n(p''_{n},q''_{n})(x)$.
	Denote $y=w_n(p''_{n},q''_{n})(x)$ and, if $y$ is  locked by $p''_{n}$ or $q''_{n}$, proceed as follows.
	Let $z_1\leq y\leq z_2$, and $f\in \{p''_{n} , (p''_{n})^{-1}, q''_{n} , (q''_{n})^{-1} \}$ be such that
	$f(z_1)=z_2$. Take
	$(A''_{n+1},p''_{n+1}, q''_{n+1})$ to be an extension of $(A''_n,p''_n,q''_n)$ such that $|A''_{n+1}\setminus A''_n|\leq 1$ and $f(y)$ is defined; clearly, $z_2<f(y)$. Let $u=s$ if $f=p''_n$, let  
	$u=s^{-1}$ if $f=(p''_{n})^{-1}$, let  $u=t$ if $f=q''_n$ , and let  $u=q^{-1}$ if $f=(q''_{n})^{-1}$, and set  
	$w_{n+1}(s,t)=uw_n(s,t)$.
	Clearly, $w_{n}(p''_n,q''_n)(x)=w_{n}(p''_{n+1},q''_{n+1})(x)< w_{n+1}(p''_{n+1},q''_{n+1})(x)$.
	Since $p'$ and $q'$ are finite and the sequence $(w_{n}(p''_n,q''_n)(x))_n$ is increasing,
	after finitely many steps we will obtain $N$ such that $w_N(p''_N,q''_N)(x)$ is not locked by 
	$p''_N$ or is not locked by $q''_N$. Set $(A'',p'',q'')=(A''_N,p''_N,q''_N)$, $w=w_N$,  and $y=w_N(p'',q'')(x)$.
\end{proof}

\begin{corollary}\label{indep}
	Suppose that $\KK$ is a \fra class, and let $\f$ be a full order expansion of $\KK$. Then $\f_2$ has no WAP.
\end{corollary}

\begin{corollary}\label{poset}
	The class $\mathcal{P}_2$ has no WAP 
\end{corollary}

\begin{proof}
	Let $P=(P,\prec^P, <^P)$ be an ordered poset, and let $p$ be a partial automorphism of $P$.
	Let $(Q,<^Q)$ be an  extension of $(P,<^P)$ such that $|Q\setminus P|=1$.
	Let $x\in P$, $y\in Q\setminus P$, and suppose that $q=p\cup\{(x,y)\}$ is a partial automorphism of
	$(Q,<^Q)$. Then there is $\prec^Q$ extending $\prec^P$ such that $(Q, \prec^Q, <^Q)$ is an ordered poset, and $q$
	is a partial automorphism of $(Q,  \prec^Q, <^Q)$. Indeed, define $\prec^Q$ as follows: for $a<^Q y$, $a\in P$, we set $a\prec^Q y$ iff there is
	$a\leq^Q b<^Q y$,  $b\in \rng(p)$,
	such that (1) $a \prec^P b$ if $a\neq b$ and (2) $\alpha^{-1}(b) \prec^P x$. 
	
	Similarly, for $y<^Qd$, $d\in P$, we set $y\prec^Q d$ iff 
	there is $y<^Q c\leq^Q d$, $c\in\rng(p)$, such that (1) $x\prec^P \alpha^{-1}(c)$ and  (2) $c\prec^P d$ if $c\neq d$.
	
	Thus, assumptions of Proposition \ref{nowap} are satisfied.
\end{proof}

It is straightforward to verify that $\p_2$ has JEP. Therefore Corollary \ref{poset} implies that
the automorphism group of the universal  ordered poset has all   $2$-dimensional diagonal conjugacy classes meager.

We show that there is no comeager $2$-dimensional diagonal conjugacy class in the automorphism group of the universal ordered boron tree.
In fact, since $\SB_2$ has JEP, which is not hard to check, this will imply that all 
 $2$-dimensional diagonal conjugacy classes of the group are meager.

\begin{theorem}\label{nowapbor}
	The class $\SB_2$ has no WAP.
\end{theorem}



Let $(A,p)\in\SB_1$, let $x\in A$, and let $O=\{a_0, \ldots, a_n\}$ be an orbit of $p$, $n\geq 2$. Suppose that 
$O$ is increasing and meet-increasing, other 3 cases being similar. We have therefore
$a_0<_{lex} \ldots<_{lex} a_n$,
and if $t_i=\meet(a_i,a_{i+1})$, then $t_0< t_1< \ldots< t_{n-1}$. 
Now the following two claims easily follow from the definition of the relation $S$.
We will use them frequently.

{\bf{Claim 1.}}
If $a_n<_{lex} x$ with $x\in {\rm Cone}_O$ 
and $(B, q)\in\SB_1$ is an extension of $(A,p)$ such that $q(x)$ is defined, then
$t_{i-1}< \meet(x,a_n)< t_i$ implies $t_i< \meet(q(x),a_n)< t_{i+1}$, for $i=1,\ldots, n-2$,
$t_{n-2}<\meet(x,a_n)< t_{n-1}$ implies $t_{n-1}< \meet(q(x),a_n)$, and 
$t_{n-1}< \meet(x,a_n)$ implies $t_{n-1}< \meet(q(x),a_n)$.

{\bf{Claim 2.}}
If $x\in {\rm Cone}_O$ is such that $x<_{lex} a_0$ and $(B, q)\in\SB_1$ is an extension of $(A,p)$
such that $q^{-1}(x)$ is defined, then $\meet(q^{-1}(x), a_0))< t_0$ (in particular, $q^{-1}(x)\notin \Cone_O$).

\smallskip

A point $x\in A$ is {\em locked} by $O$ if  
for every extension $(B,q)$ of $(A,p)$ such that $q^{-1}(a_0)$ and $q(a_m)$ are defined,
$x$ belongs to the $\leq_{lex}$-interval with endpoints  $q^{-1}(a_0)$ and $q(a_m)$. 
It is {\em locked} by $p$ if it is locked by some orbit of $p$.
This definition of locked, which we will use only in this section to discuss $\SB$, is slightly different than the one  we used earlier in  this section.
A point $x\in A$ is {\em cone-locked} by the cone $C_O$, if it is contained in $C_O$,
and it is locked or meet-locked by $O_A$. 
Finally, say that a point $x\in A$ is {\em cone-locked} by $p$ if it is
cone-locked by some $C\in\Cone(p)$.

\begin{proof}[Proof of Theorem \ref{nowapbor}]
	Let us start with some observations.  Take $(A,p)\in\SB_1$ such that every orbit has at least 3 points and let $a\in A$. If $a$ is not cone-locked by a cone from $\Cone(p)$,
	then there are extensions $(B,q), (C,r)\in\SB_1$ of $(A,p)$ such that $q(a)<_{lex} a$ and $a<_{lex} r(a)$.  
	To see this, simply take $v$,  the immediate predecessor of $a$ in $T_A$ with respect to $<$, add a new point $b$ to obtain $B$ such that  $b$ is the immediate predecessor of $a$ in $B$ with respect to $<_{lex}$,  $v<\meet(b,a)< a$, and 
	$q(a)=b$. 
	The claim below shows that this gives the required $(B,q)$. We will then similarly define $(C,r)$.
	\medskip
	
	{\bf{ Claim.}} The map $q$ preserves $S$ 
	\begin{proof}[Proof of the Claim.]
		Since for every $x\in A$, $\meet(b,x)=\meet(a,x)$, equivalently, we have to show the following. 
		
		\noindent $(\triangle)$ For every $a,b\in A$, we have
		$S(a,b,c)$ iff $S(p(a), p(b),c)$, and similarly,  $S(a,c,b)$ iff $S(p(a), c,p(b))$ and $S(c,a,b)$ iff $S(c, p(a), p(b))$.
		

		We have to consider a number of cases. Denote $m=\meet(a,p(a))$, $n=\meet(b,p(b))$,
		$m_1=\meet( p(a), p^2(a))$ and $n_1=\meet( p(b), p^2(b))$
		(if necessary, extend $(A,p)$ so that $p^2(a)$ and  $p^2(b)$ are  defined). Let $O_a$ and $O_b$ be orbits to which $a$ and $b$ belong, respectively. Without loss of generality, suppose that $O_a$ is increasing and meet-increasing.
		We will frequently use the following simple observations:
		
		(i) If $x<_{lex} a$ and  $m< \meet(x,a)$ or $a <_{lex} x <_{lex} p^2(a)$, then $x$ is $O_a$ locked.
		
		(ii) If $p(a)<_{lex} x$ and $m< \meet(x, p(a))< m_1$, then $x$ is meet-locked by $O_a$.
		
		(iii) If $c< \meet(m, n)$ then $(\triangle)$ holds.
		
		We have to consider the following cases.
		
		\begin{enumerate}
			\item It holds $b <_{lex} a$ and $\meet(b,a)< m$. In that case, we can have (a) $p(b) <_{lex} b$, or
			(b) $b<_{lex} p(b)$,  and $\meet(p(b),a)< m$, or (c) $b<_{lex} p(b)$,  and $m<\meet(p(b),a)$,
			in which case, by (i),  $O_a$ and $O_b$ intertwine.
			
			\item It holds $b <_{lex} a$ and $m< \meet(b,a)$,  or $a <_{lex} b <_{lex} p(a)$, in which case,
			by (i),  $O_a$ and $O_b$ intertwine.
			
			\item It holds $p(a)<_{lex} b$ and $m<\meet(p(a),b)$. In this case, $p^2(a)<_{lex} p(b)$.  
			We can have that: 
			
			(a) It holds $p(a)<_{lex} b <_{lex} p^2(a)$, in which case $O_a$ and $O_b$ intertwine.
			
			(b) It holds  $p^2(a)  <_{lex} b$ and $m_1< \meet(b, p^2(a))$, in which case $p^2(a)  <_{lex} p(b)$ and 
			$m<\meet(p(b), p^2(a))$. In fact, we have $m_1<\meet(p(b), p^2(a))$, as otherwise $O_b$ would be 
			an increasing orbit meet intertwining with $O_a$, which is impossible.
			
			(c) It holds $p^2(a)  <_{lex} b$ and  $m<\meet(b, p^2(a))< m_1$, in which case, by (ii), $O_a$ and $O_b$ meet intertwine and hence $m_1<\meet(p(b), p^2(a))$.

			\item     It holds $p^2(a)  <_{lex} b$ and $\meet(b, p^2(a))< m$. Then $p^2(a)  <_{lex} p(b)$ and either 
			$\meet(p(b), p^2(a))< m$ (with $b<_{lex} p(b)$ or $p(b)<_{lex} b$) or  
			$m<\meet(p(b), p^2(a))< m_1$, in which case $O_a$ and $O_b$ meet intertwine.
		\end{enumerate}
		
		This reduces checking to the following cases.
		
		Case 1: $O_a$ and $O_b$ intertwine. Without loss of generality, $a <_{lex} b <_{lex} p(a)$ (meaning that, if  instead
		$a <_{lex} p^n(b) <_{lex} p(a)$, for  $n=1$ or $n=2$, then the reasoning will be essentially the same).
		This has two subcases:
		(1a) $m< \meet(a,b)$, in which case $m_1<\meet(p(a),p(b))$ and 
		(1b)  $m<\meet(b, p(a))$, in which case $m_1<\meet(p(b), p^2(a))$. 
		
		Taking into account (i), (ii) and (iii), all we have to do is to  directly verify that $(\triangle)$ holds in (1a) and (1b) for a $c$ such that
		$p^2(a)  <_{lex} c$ and $m_1< c$.
		
		Case 2: $O_a$ and $O_b$ meet intertwine. Without loss of generality,  $m< b< m_1$ (again, meaning that, if  instead
		$m < p(b) < m_1$, then the reasoning will be essentially the same).
		
		Taking into account (i), (ii) and (iii), all we have to do is to  directly verify that $(\triangle)$ holds when
		$p^2(a) <_{lex} c <_{lex} p(b)$ and either $m_1<\meet(c,p^2(a) )$ or $m_1<\meet(c,p(b) )$.

		Case 3: $b, p(b) <_{lex} a$ and $\meet(b, a), \meet(p(b),a)< m$.
		
		Then possibilities on $c$ are: (3a) $b, p(b)<_{lex} c<_{lex} a$ and $\meet(c, a)< m$,
		(3b) $p(a)<_{lex}~c$ and $m<\meet(p(a),c)$,
		(3c) $p(a)<_{lex} c$ and $\meet(b,a),\meet(p(b),a)<\meet(p(a),c)< m$
		(3d) $p(a)<_{lex} c$ and $\meet(p(a),c)$ is between $\meet(b,a)$ and $\meet(p(b),a)$
		(this cannot happen though, otherwise $c$ would be meet-locked by $O_b$).

		Case 4: $p(a) <_{lex}   b, p(b)$ and $\meet(b, p(a)), \meet(p(b),p(a))< m$.
		This is very similar to Case 3.
		
		Case 5:  $p^2(a) <_{lex}   b, p(b)$ and $m_1< \meet(b, p^2(a)), \meet(p(b),p^2(a))$.
		Let $p=\max\{\meet(b, p^2(a) ), \meet(p(b), p^2(a) )\}$.
		Then possibilities on $c$ are: 
		(5a) $p^2(a) <_{lex}  c <_{lex}  b$ and $p<\meet(c, p^2(a) )$,
		(5b) $p^2(a) <_{lex}  c <_{lex}  b$ and $\meet(c, p^2(a) )\leq p$,
		(5c) $b,p(b)<_{lex}  c $ and $m_1<  \meet(c, p(b))< n$.	
\end{proof}

	Now we show that for a given 
	$(A,p,q)\in\SB_2$ and $x\in A$ such that $x< p(x)$ 
	there exists $(A', p',q')\in\SB_2$, an extension of $(A,p,q)$, and a word $w(s,t)\in F_2$, such that  $w(p',q')(x)$ is defined and not
	cone-locked by $p'$ or $q'$. Then an argument  presented in the first paragraph of the proof of Proposition \ref{nowap}, will finish the proof. Without loss of generality, every non-trivial orbit of $p$ and $q$ consists of at least three points.
	
	As for any $(A,p)\in\SB_1$, $A$ is a substructure of the \fra limit $M$ of $\SB$, we  consider 
	\[cl_{p}=\{x\in M\colon x \text{ is cone-locked by an orbit of } p\}.\]
	Note that for every orbit $O$ of $p$, the set $\{x\in M\colon x \text{ is cone-locked by } O\}$ is the union of two 
	$\leq_{lex}$-intervals,	one of them constituted  of points locked by $p$, and the other one of points meet-locked by $p$.
	This implies that $cl_{p}$ is the union of at most $2m_{p}$ disjoint $\leq_{lex}$-intervals, where $m_{p}$ is the number of orbits in $p$.
	Denote this collection of $\leq_{lex}$-intervals by
	$\mathcal{I}_{p}$, and  its cardinality  by $n_{p}$. Observe that the following hold:
	
	$(*)$  For every $I\in \mathcal{I}_{p}$ and $x\in I$, there is an extension $(A',p')$ of $(A,p)$ so that $(p')^m(x)<_{lex} I<_{lex}(p')^n(x)$ for some $m,n\in \mathbb{Z}$.

	$(**)$ For every $(A',p') \in\SB_1$ extending $(A,p)$ with 
	$A'\setminus A=\{(p')^{\epsilon}(a),\ldots, (p')^{\epsilon k}(a)\}$ for some $a\in A$,
	$\epsilon\in\{-1,1\}$, and $k\in\mathbb{N}$, and for every  $I\in \mathcal{I}_{p'}$, there is 
	$J\in \mathcal{I}_{p}$ such that $J\subseteq I$. In particular, $n_{p'}\leq n_{p}$.

	We construct the required $(A',p',q')$ and $w$ inductively.
	Let $(A'_0,p'_0,q'_0)=(A,p,q)$ and $w_0=1$. Suppose that we already constructed $(A'_n,p'_n,q'_n)$ and $w_n(s,t)$  and 
	suppose  that $w_{n-1}(p'_{n-1},q'_{n-1})(x)=w_{n-1}(p'_{n},q'_{n})(x)<_{lex} w_n(p'_{n},q'_{n})(x)$.
	Denote $y=w_n(p'_{n},q'_{n})(x)$ and if $y$ is  cone-locked by $p'_{n}$ and $q'_{n}$, proceed as follows.
	Let $I_{p,y}\in  \mathcal{I}_{p'_{n}}$ be the $\leq_{lex}$-interval containing $y$, and similarly define
	$I_{q,y}$. If the right endpoint of $I_{p,y}$ is $<_{lex}$-greater or equal than the right endpoint of $I_{q,y}$, there must exist $k\in\mathbb{Z}$ such that $y<_{lex} (q'')^k(y)\notin I_{p,y}$ in some extension $(A'',p'',q'')$ of $(A'_n,p'_n,q'_n)$. Take the smallest such $k$, set $w_{n+1}(s,t)=t^kw_n(s,t)$ and let $(A'_{n+1},p'_{n+1},q'_{n+1})$ be a minimal such extension. This can be done by observation $(*)$.
	Similarly,  If the right endpoint of $I_{q,y}$ is $<_{lex}$-greater than the right endpoint of $I_{p,y}$, there must exist $k\in\mathbb{Z}$ such that, in some extension  $(A'',p'',q'')$, we have $y<_{lex} (q'')^k(y)\notin I_{q,y}$. Take the smallest such $k$, set $w_{n+1}(s,t)=s^kw_n(s,t)$ and let $(A'_{n+1},p'_{n+1},q'_{n+1})$ be a minimal such extension. 
	
	Observation $(**)$ implies that after at most $n_{p}+n_q$ many steps,
	this construction will stop, i.e., that for some  $n\leq n_{p}+n_q$, we will have that 
	$w_n(p'_n,q'_n)(x)$ is not cone-locked by $p'_n$ or by $q'_n$. 
\end{proof}



Directed ultrahomogeneous graphs were classified
by Cherlin \cite{C}  and their precompact Ramsey expansions were described by Jasi\'nski-Laflamme-Nguyen van Th\'e-Woodrow \cite{JLNW}. See  page 74 in \cite{C} for the list of them, see also page 5 in~\cite{JLNW}.
By Kechris-Pestov-Todorcevic \cite{KPT} and Nguyen Van Th\'e \cite{NVT}, automorphism groups of
those expansions are extremely amenable. The notation we will use comes from \cite{JLNW}.

\begin{theorem}\label{direct}
Let $M$ be a 
	precompact Ramsey expansion of a directed ultrahomogeneous graph, and let $\f=\age(M)$.
	Then $\f_2$ does not have WAP.
\end{theorem}

\begin{proof}[Sketch of the proof]
			
Proposition \ref{nowap} applies to the age of each of these structures.			
A number of those structures are directly taken care of by Corollary \ref{indep}. These are rational numbers and 
precompact Ramsey expansions 
of: the random tournament $T^\omega$, $\Gamma_n$-- the random directed graph that does not embed the edgeless graph on $n$ vertices, $n\leq\omega$, $\mathcal{T}$ -- the random directed graph that does not embed finite
tournaments from some fixed set $\mathcal{T}$.
Moreover, structures $S(2)^*$ and $S(3)^*$, that is,
precompact Ramsey expansions of  $S(2)$ and $S(3)$,
are first-order simply bi-definable to structures $Q_2$ 
and $Q_3$, discussed by Nguyen van Th\'e \cite{NVT}, whose age is of the form as in Corollary \ref{indep}.

Furthermore, proofs for the precompact Ramsey expansion of the structures of the form $T[I_n]$, $I_n[T]$, where $I_n$ is the edgeless graph on $n$ vertices, $n\leq \omega$, and $T$ is a homogeneous tournament, as well as  of $\widehat{\mathbb{Q}}$,
$\widehat{T^\omega}$ and of the complete $n$-partite random directed graph, $n\leq\omega$,  are  essentially the same as those for the structures taken care of by Corollary~\ref{indep}
(i.e. perhaps there are some additional unary predicates, which do not change the proof in an essential way). Let us discuss
here one of these structures.
We describe ${\rm Age}(T[I_n]^*)$ of the expansion $T[I_n]^*$ of $T[I_n]$, 
where $T$ is a generic
tournament and $I_n$ is the edgeless  directed graph on $n< \omega$
vertices with the usual ordering (inherited from the natural numbers).
Consider the language $L=\{E,<, L_1,\ldots, L_n\}$, where $E,<$ are binary predicates and $L_1,\ldots, L_n$ are unary predicates. We will use $E$ for the edge relation and $<$  for the linear
order. 
We let the ${\rm Age}(T[I_n]^*)$ to consist of substructures of structures whose universe is of the form $S\times I_n$, where $S$ is a linearly ordered tournament (the choice of a linear ordering is arbitrary). A pair $((x,i),(y,j))$ is an edge in $S\times I_n$ iff the pair $(x,y)$ is an edge in $S$. The ordering we put on  $S\times I_n$ is lexicographic with respect to the order on $S$ and on $I_n$. Finally, we set $L_i(x,j)$ iff $i=j$.
 It is clear how to modify the proof from Corollary \ref{indep} to prove
that assumptions of Proposition \ref{nowap} are satisfied for ${\rm Age}(T[I_n]^*)$ as well.
			
We have already discussed the universal ordered poset in Corollary \ref{poset}.
The precompact Ramsey expansion $\mathcal{P}(3)^*$ of the `twisted' universal ordered poset $\mathcal{P}(3)$ is  first-order simply bi-definable to the \fra limit of the family $\mathcal{K}_0$ of ordered posets, additionally
equipped with~3  subsets (described using unary predicates) forming a partition of the universe of the ordered poset,
see the bottom of the page 21  in \cite{JLNW}.
			
Proposition \ref{nowap} also applies to the age of $\mathcal{S}^*$, the precompact Ramsey expansion  of the
semigeneric directed graph $\mathcal{S}$, which is rather straightforward to check.
In fact, for given  $(A,p)$  and $a\in A$ not locked by $p$,
the required $q(a)$ and $r(a)$ (notation taken from the statement of Proposition \ref{nowap})
can be chosen in the same equivalence class with respect to  the non-edge equivalence relation in which $a$ is.

\end{proof}

\section{The two-dimensional case. Similarity classes}

Slutsky \cite{Sl} showed that every $2$-dimensional topological similarity class in $\Aut(\mathbbm{Q})$ is meager. In this section, we extract from Slutsky's arguments a general condition on a structure $M$ that implies that every $2$-dimensional topological similarity class in $\Aut(M)$ is meager (Theorem \ref{th:MeagerSim}). 

Let $\f$ be a \fra class. Generalizing the terminology introduced in \cite{Sl}, for $A,B \in \f$ with $B \subseteq A$, and a partial automorphism $q$ of $A$ such that $ \zdef(q) \cap B=\emptyset$, we say that $B$ is \emph{free from} $q$ if for every $n$, every relation symbol $R$ in the signature of $\f$ of arity $n$, for all $x_1, \ldots, x_n \in B \cup \dom(q)$ we have $R(x_1, \ldots, x_n)$ iff  $R(y_1, \ldots, y_n)$, where $y_i=x_i$ if $x_i \in B$, and $y_i=q(x_i)$ if $x_i \in \dom(q)$. In other words, we can extend $q$ so that $q(x)=x$ for every $x \in B$.

We say that $\f$ has \emph{liberating automorphisms} if for any partial automorphisms $p,q$ of $A \in \f$ with no cyclic orbits there exists $N \in \NN$ such that, for every $N'>N$, $p$ can be extended to a partial automorphism $p'$ of an element of $\f$ so that $(p')^n[A]$ is free from $q$ for all $n$ with $N \leq n \leq N'$. 
 
Let $\f$ be an \fra order class with an order relation $<$, and let $M$ be the limit of $\f$. Let $p$ be a partial automorphism of $M$. For a convex $A \subseteq M$ (i.e., $x,y \in A$, and $x<z<y$ entails that $z \in A$), we say that $p$ is \emph{increasing on} $A$ if for every $x \in A$, $p$ can be extended so that $p(x)>x$; it is \emph{decreasing on} $A$ if for every $x \in A$, $p$ can be extended so that $p(x)<x$; it is \emph{monotone on} $A$ if it is increasing or decreasing on $A$.  We say that an extension $p'$ of $p$ \emph{does not change monotonicity} of $p$ if there are no new fixed points in $p'$, and $p'$ is increasing (decreasing) on $(a,b)$ iff $p$ is increasing (decreasing) on $(a,b)$, for $a,b \in \zdef(p)$.
We say that $p$ is \emph{eventually increasing} if there exist $x,y \in \supp(p)$ such that $z<p(z)$ for every $z \in \supp(p)$ such that either $z \leq x$ or $y \leq z$ (i.e., the first and the last orbits of $p$ are increasing.) 


Let $(p,q)$ be a pair of partial automorphisms of $M$. We say that $x$ is in a \emph{final segment} (or \emph{initial segment}) of $(p,q)$ if there exists a common fixed point $y$ of $p$ and $q$ such that $p$ is monotone on $[x,y)$ (or $(y,x]$.) We say that $(p,q)$ is \emph{elementary} if both $p$ and $q$ are eventually increasing, and the only common fixed points of $p$ and $q$ are the minimum $\min(\dom(p))=\min(\dom(q))$, and the maximum $\max(\dom(p))=\max(\dom(q))$ of their domains. A pair $(p,q)$ is \emph{piecewise elementary} if we can find $E_0 \leq \ldots \leq E_n$, called \emph{elementary components} of $(p,q)$, such that $\bigcup_i E_i=\zdef(p) \cup \zdef(q)$, and $(p,q)$ is elementary when restricted to each $E_i$.

\begin{lemma}
\label{le:ElPair}
Let $\f$ be a full order expansion with SAP. Let $(p,q)$ be an elementary pair. Then there exists an extension $(p',q')$ of $(p,q)$, and $w \in F(s,t)$, such that $p'$ does not change monotonicity of $p$, and $w(p',q')[\zdef(q')]$ is in the unique final segment of $(p',q')$.
\end{lemma}

\begin{proof}
Without loss of generality, we can assume that $\min(\supp(p))<\supp(q)$. Let $a_0< \ldots < a_{m}$ be the enumeration of $\zdef(p) \setminus \{\min(\zdef(p)), \max(\zdef(p))\}$. We can assume that $a_0$ is not a fixed point of $p$.

We construct $w_i \in F(s,t)$, $i\leq m$, and extensions $(p_i,q_i)$ of $(p_{i-1},q_{i-1})$ such that $w_i \ldots w_0(p_i,q_i)(a_0)>a_i$. Moreover, we require that 
the only new element in $\zdef(p_i)$ or $\zdef(q_i)$ above $a_{i}$ is $w_i \ldots w_0(p_i,q_i)(a_0)$. 

Put $w_{-1}=\emptyset$, $p_{-1}=p$, $q_{-1}=q$. Fix $0 \leq i \leq m$, and suppose that $w_j, p_j,q_j$ have been already constructed for $j<i$. Set $b=w_{i-1}\ldots w_0(p_{i-1},q_{i-1})(a_0)$. If there is an extension $p_i$ of $p_{i-1}$ such that $(p_i)^{\epsilon k}(b)>a_{i}$ for some $k \in \NN$, and $\epsilon \in \{-1,1\}$, we take the least such $k$, and put $w_i=s^{\epsilon k}$, $q_i=q_{i-1}$.

Suppose otherwise. 
If $a_{i+1}$ is not a fixed point of $q$, put $c=a_{i}$, and $l=0$. Otherwise, as $(p_{i-1}, q_{i-1})$ is elementary, it is not a fixed point of $p_{i-1}$, and so we can extend $p_{i-1}$ to some $p_i$ by adding only elements below $a_{i}$, so that, for some $l \in \ZZ$, $b<(p_i)^l(a_{i})<a_{i}$, and $(b,(p_i)^l(a_{i}))$ has empty intersection with both $\zdef(p_{i-1})$ and $\zdef(q_{i-1})$.
Put $c=(p_i)^l(a_{i})$.

Let $\epsilon \in \{-1,1\}$ be such that there exists an extension $q_i$ of $q_{i-1}$ with $(q_i)^{\epsilon}(c)<c$. 
Because $\f$ is a full order expansion with SAP, there is an extension of $p_i$, which we will also denote by $p_i$, such that $p_i(b) \in ((q_i)^{\epsilon}(c),c)$. But then $(q_i)^{-\epsilon}(p_i(b))>c$. Thus, 
for $w_i=s^{-l}t^{-\epsilon}s$, we have that $w_i\ldots w_0(p_i,q_i)(a_0)>a_{i}$ and $w_i\ldots w_0(p_i,q_i)(a_0)$ is the only new element in $\zdef(p_i)$ or $\zdef(q_i)$ above $a_{i}$.

\end{proof}


\begin{lemma}
\label{le:Pom}
Let $\f$ be a full order expansion with SAP, and such that $\f^-$ has liberating automorphisms. Let $M$ be the limit of $\f$. Let $(p,q)$ be a piecewise elementary pair of partial automorphisms of $M$ such that, for some $w \in F(s,t)$, $w(p,q)(x)$ is in a final segment of $(p,q)$ for every $x \in \zdef(q)$. Then for any $v \in F(s,t)$, and $N \in \NN$, there is $N' \geq N$, and a pair $(p',q')$ extending $(p,q)$ such that $vs^{N'}w(p',q')(x)$ is in a final segment of $(p,q)$ for $x \in \zdef(q')$.
\end{lemma}

\begin{proof}
Let $E_0 \leq \ldots  \leq E_n \subseteq M$ be elementary components of $(p,q)$. Because $w(p,q)(x)$ is in a final segment of $(p,q)$ for every $x \in \zdef(q)$, and $p$ is eventually increasing on each $E_i$, we can assume that there is $N_0 \in \NN$ such that
\[ s^{N_0}w(p,q)(x)>\supp(q) \cap E_i \]
for every $i \leq n$, and $x \in \supp(q) \cap E_i$. Then 
\begin{equation}
\label{eq:1}
x \leq s^{N_0}w(p,q)(y) \mbox{ if and only if } q(x) \leq s^{N_0}w(p,q)(y)
\end{equation}
for every $x,y \in \zdef(q)$.

Because $\f^-$ has liberating automorphisms, and $\f$ is a full order expansion with SAP, we can find $N_1 \in \NN$, and an extension $p'$ of $p$ such that $s^{N_0+N_1+n}w(p',q)[\zdef(q)]$ is free from $q$ in $\f$ for $n \leq 2 \left| v \right|+N$, and (\ref{eq:1}) still holds for $\zdef(q)$. But this means that $s^{N_0+N_1+n}w(p',q)[\zdef(q)]$ is free from $q$ in $\f$, and we can extend $q$ to $q'$ so that
\[ q'(s^{N_0+N_1+n}w(p',q')(x))=s^{N_0+N_1+n}w(p',q')(x) \]
for $n \leq 2\left| v \right|+N$ and $x \in \zdef(q)$. It is easy to see that then
\[ vs^{N_0+N_1+\left|v \right|+N}w(p',q')(x) \geq w(p',q')(x) \]
for $n \leq N$, and so $vs^{N_0+N_1+\left|v \right|+N}w(p',q')(x)$ is in a final segment of $p'$, for $x \in \zdef(q')$. Therefore $N'=N_0+N_1+\left| v \right|+N$ is as required. 
\end{proof}

\begin{lemma}
\label{le:PieceElPair}
Let $\f$ be a full order expansion with SAP, and such that $\f^-$ has liberating automorphisms. Then for every piecewise elementary pair $(p,q)$ there exists a piecewise elementary pair $(p',q')$ extending $(p,q)$, and $w \in F(s,t)$ such that $w(p',q')(x)$ is in a final segment of $(p',q')$ for $x \in \zdef(q')$.
\end{lemma}

\begin{proof}
We prove the lemma by induction on the number $r$ of elementary components of $(p,q)$. For $r=1$, this follows from Lemma \ref{le:ElPair}. 
Suppose that that the lemma is true for some $r$, and fix a piecewise elementary pair $(p,q)$ with $r+1$ elementary components.

Let us write $E=E_0 \cup E_1$ so that $E_0 \leq E_1$, $(p,q)$ is elementary when restricted to $E_0$, and there are $r$ elementary components in $(p,q)$ when restricted to $E_1$. Using Lemma \ref{le:ElPair}, and the inductive assumption, we can fix an extension $(p',q')$ of $(p,q)$, so that, for $(p_0,q_0)$ denoting the restriction of $(p',q')$ to $E_0$, and $(p_1,q_1)$ denoting the restriction of $(p',q')$ to $E_1$, the following holds. The mapping $p_0$ does not change monotonicity of $p$ restricted to $E_0$, the pair $(p_1,q_1)$ is piecewise elementary, and there exist $w_0,w_1 \in F(s,t)$ such that $w_0(p_0,q_0)(x)$ is in the unique final segment of $(p_0,q_0)$ for $x \in \zdef(q_0)$, and $w_1(p_1,q_1)(x)$ is in a final segment of $(p_1,q_1)$ for $x \in \zdef(q_1)$. 

Let $d_0=\min(\supp(p_0))$. As $p_0$ does not change monotonicity of $p$ restricted to $E_0$, and so $p_0$ is increasing on the initial segment of $p_0$, in the case that $w_1(p_0,q_0)(d_0)<d_0$, we can assume that there exists $N \in \NN$ such that $s^N w_1(p_0,q_0)(d_0)>d_0$. Then $w_0 s^N w_1(p_0,q_0)(x)$ is in the final segment of $p_0$ for $x \in \zdef(q_0)$.

Moreover, applying Lemma \ref{le:Pom}, we can assume that $N$ is large enough so that $w_0s^N w_1(p_1,q_1)(x)$ is in a final segment of $p_1$ for $x \in \zdef(q_1)$. Thus, the pair $(p',q')$ extends $(p,q)$, and, $w(p',q')(x)$ is in a final segment of $p'$ for $x \in \zdef(q')$, if $w=w_0s^N w_1$.
\end{proof}

\begin{theorem}
Let $\f$ be a full order expansion with SAP, and such that $\f^-$ has liberating automorphisms. Let $M$ be the limit of $\f$. Then there are comeagerly many pairs in $\Aut(M)^2$ generating a non-discrete group.
\end{theorem}

\begin{proof}
Consider the following condition: for every pair $(p,q)$ of partial automorphisms of $M$ there exists an extension $(p',q')$ of $(p,q)$, and $w \in F_2$ such that $w(p',q')(x)=x$ for $x \in \zdef(q)$. It is easy to verify that if it holds, then the set of pairs $(f,g) \in \Aut(M)^2$ generating a non-discrete group contains a dense $G_\delta$ subset of $\Aut(M)^2$, that is, it is comeager in $\Aut(M)^2$.

We verify this condition. Fix a pair $(p,q)$ of partial automorphisms of $M$. Without loss of generality, we can assume that it is piecewise elementary. By Lemma \ref{le:PieceElPair}, we can also assume that there exists $w' \in F(s,t)$ such that $w'(p,q)(x)$ is in a final segment of $p$ for $x \in \zdef(q)$. But then, using our assumptions on $\f$, and the fact that $p$ is eventually increasing on each elementary component of $(p,q)$, we can find $N \in \NN$, and an extension $(p',q')$ of $(p,q)$ such that
\begin{equation}
\label{eq:2}
x \leq s^{N}w(p',q')(y) \mbox{ if and only if } q(x) \leq s^{N}w(p',q')(y)
\end{equation}
for every $x,y \in \zdef(q')$, and $s^N w'(p',q')[\zdef(q')]$ is free from $q'$ in $\f$. Therefore we can put
\[ q'(s^N w(p',q')(x))=s^N w(p',q')(x) \]
for $x \in \zdef(q')$. Then for $w=(s^N w)^{-1}t(s^N w)$, we have $w(p',q')(x)=x$ for $x \in \zdef(q')$.

\end{proof}

\begin{theorem}
\label{th:MeagerSim}
Let $\f$ be a full order expansion with SAP, and such that $\f^-$ has liberating automorphisms. Let $M$ be the limit of $\f$. Then every  $2$-dimensional topological similarity class in $\Aut(M)$ is meager.
\end{theorem}

\begin{proof}
By the above theorem, there are comeagerly many pairs in $\Aut(M)^2$ generating a non-discrete group. As automorphisms of order structures have only infinite non-trivial orbits, in fact, there are comeagerly many pairs in $\Aut(M)^2$ generating a non-discrete and non-precompact group. By \cite[Theorem 4.4]{KwMa}, every $2$-dimensional topological similarity class in $\Aut(M)$ is meager.
\end{proof}

Recall that a \fra class $\f$ has \emph{free amalgamation} if for every $A,B,C \in \f$ with $A \subseteq B,C$, the structure $D=B \cup C$ is an amalgam of $B$ and $C$ over $A$. In other words, no tuple in $D$ involving at the same time elements from $B \setminus A$ and from $C \setminus A$ is related in $D$. A typical example of a class with free amalgamation is the class of finite graphs.

\begin{lemma}
\label{le:FreeAm}
If $\f$ is a \fra class with free amalgamation or the class of finite tournaments, then $\f$ has liberating automorphisms.
\end{lemma}

\begin{proof}
Suppose that $\f$ has free amalgamation. Let $(p,q)$ be a pair of partial automorphisms of $A \in \f$ with no cyclic orbits. Without loss of generality, we can assume that $A=\zdef(p)$. Let $N \in \NN$ be such that orbits in $p$ have size at most $N/2$. Set $C=\dom(p) \setminus \rng(p)$, and fix $N'>N$. We put $p_0=p$, and construct partial automorphisms $p_i$ and sets $D_i$, $0<i \leq N'$ in the following manner. Assuming that $p_i$ is already constructed, with an aid of free amalgamation, we extend $p_i$ to a partial automorphism $p_{i+1}$ by defining it on $D'=\rng(p_i) \setminus \dom(p_i)$ in such a way that no relation involves at the same time elements from $C$ and $D=p_{i+1}[D']$ (where $D$ is disjoint from $\zdef(p_i)$.) To be more precise, put $B =\rng(p_i)$. Then, for every relation $R$ of arity $n$, and every $n$-tuple $\bar{b}$ in $B \cup D$, whether $R(\bar{a})$ holds or not, is entirely determined by the requirement that $p_{i+1}$ is supposed to be a partial automorphism.  Moreover, regardless of how we amalgamate $B \cup C$ and $B \cup D$ over $B$, to get a structure $E$ with underlying set $B \cup C \cup D$, $p_{i+1}$ will be a partial automorphism of $E$. Thus, $E$ obtained by freely amalgamating these structures works as $\zdef(p_{i+1})$. Finally, we put $D_{i+1}=D$.

Observe that no relation involves at the same time elements from $p_{N'}^{n}[A]$ and $A$ for $N \leq n \leq N'$, which means, because $\zdef(q) \subseteq A$, that each $p_{N'}^{n}[A]$ is free from $q$.  Indeed, by the construction of $p_i$, for $i>0$ and $x \in D_i$, no relation involves $x$ and elements from $C$. And then the same is true about any $p^n_{i+1}(x)$ and $p^i_{i+1}[C]$, $i \leq n \leq N'$.
As $p_{N'}^n[A] \subseteq \bigcup_{N \leq i \leq N'} D_i$, and $A \subseteq \bigcup_{i \leq N} p^i[C]$, this means that no relation involves at the same time elements from $p_{N'}^{n}[A]$ and $A$, for $N \leq n \leq N'$.

For finite tournaments, we proceed almost exactly as above. The only difference is that for every $x \in \dom(p_i) \setminus \rng(p_i)$, $y \in \rng(p_i) \setminus \dom(p_i)$, we choose $(x,y)$ as the arrow between $x$ and $y$. Then $(x,y)$ is an arrow for every $x \in \zdef(q)$ and $y \in p_{N'}^{n}[\zdef(q)]$, $N \leq n \leq N'$.
\end{proof}
 
\begin{corollary}
\label{th:MeagerSimFree}
Suppose that $\f$ is a full order expansion such that $\f^-$ has free amalgamation, or the class of finite ordered tournaments. Let $M$ be the \fra limit of $\f$. Then every  $2$-dimensional topological similarity class in $\Aut(M)$ is meager.
\end{corollary}

\begin{corollary}
Suppose that $\f$ is a full order expansion of a class with free amalgamation, and let $M$ be the \fra limit of $\f$. Then every  $2$-dimensional topological similarity class in $\Aut(M)$ is meager.
\end{corollary}

\begin{remark}
Compare the above corollary with Theorem \ref{th:NoComeager1} and Corollary \ref{indep}.
\end{remark}

\end{document}